\documentclass[12pt,reqno,english]{amsart}
\usepackage[T1]{fontenc}
\usepackage[latin9]{inputenc}
\usepackage{geometry}
\usepackage{color}
\usepackage{mathrsfs}
\usepackage{mathtools}
\usepackage{bm}
\usepackage{amsbsy}
\usepackage{amstext}
\usepackage{amsthm}
\usepackage{amssymb}
\usepackage{graphicx}
\usepackage{setspace}
\usepackage{esint}
\makeatletter
\numberwithin{equation}{section}
\numberwithin{figure}{section}
\theoremstyle{plain}
\newtheorem{thm}{\protect\theoremname}[section]
\theoremstyle{remark}
\newtheorem{notation}[thm]{\protect\notationname}
\theoremstyle{plain}
\newtheorem{prop}[thm]{\protect\propositionname}
\theoremstyle{remark}
\newtheorem{rem}[thm]{\protect\remarkname}
\theoremstyle{plain}

\theoremstyle{remark}
\newtheorem*{rem*}{\protect\remarkname}
\theoremstyle{plain}
\newtheorem{lem}[thm]{\protect\lemmaname}
\theoremstyle{remark}
\newtheorem*{acknowledgement*}{\protect\acknowledgementname}

\usepackage{babel}
\usepackage{stmaryrd}

\makeatother

\usepackage{babel}
\providecommand{\acknowledgementname}{Acknowledgement}
\providecommand{\corollaryname}{Corollary}
\providecommand{\lemmaname}{Lemma}
\providecommand{\notationname}{Notation}
\providecommand{\propositionname}{Proposition}
\providecommand{\remarkname}{Remark}
\providecommand{\theoremname}{Theorem}

\begin{document}

\title[Non-reversible Metastable Diffusions with Gibbs Invariant Measure
II]{Non-reversible Metastable Diffusions with Gibbs Invariant Measure
II: Markov Chain Convergence}
\author{Jungkyoung Lee and Insuk Seo}
\address{J. Lee. Department of Mathematical Sciences, Seoul National University,
Republic of Korea.}
\email{ljk9316@snu.ac.kr}
\address{I. Seo. Department of Mathematical Sciences and Research Institute
of Mathematics, Seoul National University, Republic of Korea.}
\email{insuk.seo@snu.ac.kr}
\begin{abstract}
This article considers a class of metastable non-reversible diffusion
processes whose invariant measure is a Gibbs measure associated with
a Morse potential. In a companion paper \cite{LeeSeo}, we proved
the Eyring--Kramers formula for the corresponding class of metastable
diffusion processes. In this article, we further develop this result
by proving that a suitably time-rescaled metastable diffusion process
converges to a Markov chain on the deepest metastable valleys. This
article is also an extension of \cite{RS}, which considered the same
problem for metastable reversible diffusion processes. Our proof is
based on the recently developed resolvent approach to metastability.
\end{abstract}

\maketitle

\section{\textcolor{black}{\label{sec1}Introduction }}

\textcolor{black}{In this article, we focus on the analysis of the
metastable behavior of a class of diffusion processes given by a stochastic
differential equation (SDE) in $\mathbb{R}^{d}$ of the form
\begin{equation}
d\bm{x}_{\epsilon}(t)\,=\,-(\nabla U+\boldsymbol{\ell})(\bm{x}_{\epsilon}(t))\,dt+\sqrt{2\epsilon}\,d\bm{w}_{t}\;,\label{e_sdex}
\end{equation}
where $\epsilon>0$ is a small constant, $U\in C^{2}(\mathbb{R}^{d})$
is a Morse function, and $\boldsymbol{\ell}\in C^{1}(\mathbb{R}^{d},\,\mathbb{R}^{d})$
is a smooth vector field satisfying two constraints, $\nabla U\cdot\boldsymbol{\ell}\equiv0$
and $\nabla\cdot\boldsymbol{\ell}\equiv0$ whose meaning will be explained
later in detail. In a companion paper \cite{LeeSeo}, we investigated
this model as a generalization of the well-known metastable reversible
overdamped Langevin dynamics given by the SDE
\begin{equation}
d\boldsymbol{y}_{\epsilon}(t)\,=\,-\nabla U(\boldsymbol{y}_{\epsilon}(t))\,dt+\sqrt{2\epsilon}\,d\bm{w}_{t}\;.\label{e_sdey}
\end{equation}
The process of the form $\boldsymbol{x}_{\epsilon}(\cdot)$ has been
investigated in many studies and applied to stochastic optimization
and Markov chain Monte Carlo, see e.g., \cite{DLP,HHS,HHS2,LeeSeo,LM,LNP,ReS1,ReS2}
and the references therein, mainly}\textbf{\textcolor{black}{{} }}\textcolor{black}{because
this model is a natural generalization of the Langevin dynamics $\boldsymbol{y}_{\epsilon}(\cdot)$
and has a Gibbs invariant measure. Moreover, it is widely believed
that the process $\boldsymbol{x}_{\epsilon}(\cdot)$ has a better
mixing property than the process $\boldsymbol{y}_{\epsilon}(\cdot)$.
In fact, this belief has been quantitatively verified in \cite{LeeSeo}
and \cite{LM} in view of the so-called Eyring--Kramers formula and
low-lying spectra, respectively. }

\subsubsection*{\textcolor{black}{Main contribution of the article}}

\textcolor{black}{The metastable behaviors of the }\textit{\textcolor{black}{reversible}}\textcolor{black}{{}
process $\boldsymbol{y}_{\epsilon}(\cdot)$, exhibited when $U$ has
multiple local minima, have attracted considerable attention in recent
decades, and their accurate quantitative analysis has been thoroughly
investigated in many studies. For instance, \cite{BEGK1,new_HKN}
established the Eyring--Kramers formula, \cite{BEGK2} provided the
sharp asymptotics of low-lying spectra, \cite{RS,new_Su1} described
the metastable behavior as a limiting Markov chain under a suitable
exponential time-rescaling, and \cite{new_GLLNreview,new_GLLNresearch,new_GLLNexitpoint,new_GLLNexitpoint2,new_LLN book,new_LNdouble well,new_Nectoux_criticalbdry}
developed the quasi-stationary distribution approach for this process.
The last approach is based on the theories from semi-classical analysis
developed in \cite{new_HKN,new_Michel}. We note that these approaches
are the most typical methods for quantitatively investigating the
metastable behavior of a metastable process. }

\textcolor{black}{One of the main features of the process $\boldsymbol{y}_{\epsilon}(\cdot)$
is the fact that it is reversible with respect to its Gibbs invariant
measure of the form
\begin{equation}
\mu_{\epsilon}(d\boldsymbol{x})\,=\,Z_{\epsilon}^{-1}\,e^{-U(\bm{x})/\epsilon}\,d\boldsymbol{x}\;,\label{e_inv}
\end{equation}
where $Z_{\epsilon}$ is the partition function given by
\begin{equation}
Z_{\epsilon}\,=\,\int_{\mathbb{R}^{d}}\,e^{-U(\bm{x})/\epsilon}\,d\bm{x}\;.\label{e_partion}
\end{equation}
Owing to this reversibility, many tools are available to investigate
the process $\boldsymbol{y}_{\epsilon}(\cdot)$. However, }\textit{\textcolor{black}{nearly
none of these tools is applicable to non-reversible processes such
as $\boldsymbol{x}_{\epsilon}(\cdot)$}}\textcolor{black}{. Hence,
the quantitative analysis of the metastability of non-reversible processes
has long been an open issue. To this end, many innovative studies
such as \cite{GL,Lan3,LMS,LS1,Seo} have been conducted in recent
years, and several non-reversible metastable processes have been analyzed.
In particular, the non-reversible process $\boldsymbol{x}_{\epsilon}(\cdot)$
has been analyzed in two recent studies. The Eyring--Kramers formula
has been proven in \cite{LeeSeo}, and low-lying spectra have been
analyzed in \cite{LM}. In the present article, we present the Markov
chain description of the metastable behavior of the process $\boldsymbol{x}_{\epsilon}(\cdot)$,
which is a highly precise description of such metastable behavior. }

\subsubsection*{\textcolor{black}{Markov chain description of metastable behavior}}

\textcolor{black}{Now, we explain why the Markov chain description
is a natural and effective description of the metastable behavior
of the process. To explain this in a more intuitive manner, we first
consider the reversible dynamics $\boldsymbol{y}_{\epsilon}(\cdot)$
studied in \cite{RS}. Consider this process as a small random perturbation
of the dynamical system given by an ordinary differential equation
(ODE) of the form
\begin{equation}
d\boldsymbol{y}(t)\,=\,-\nabla U(\boldsymbol{y}(t))\,dt\;.\label{e_odey}
\end{equation}
Note that a local minimum of the potential function $U$ is a stable
equilibrium of the dynamics $\boldsymbol{y}(\cdot)$. Now, suppose
that $U$ has multiple global minima as shown in Figure \ref{fig1}}\footnote{\textcolor{black}{This figure has been excerpted from \cite[Figure 1.2]{RS}.}}\textcolor{black}{{}
and that the process $\boldsymbol{y}_{\epsilon}(\cdot)$ starts from
a small neighborhood of a minimum, which is called a (metastable)
}\textit{\textcolor{black}{valley}}\textcolor{black}{. In the first
stage, the process $\boldsymbol{y}_{\epsilon}(\cdot)$ stays in this
valley for a long time because of the strong force $-\nabla U(\boldsymbol{y}_{\epsilon}(t))dt$
toward the minimum. However, the small noise term $\sqrt{2\epsilon}\,d\boldsymbol{w}_{t}$
will finally push the process toward another valley after an exponentially
long time, and this can be understood via the large-deviation principle
(cf. \cite{FW}). This movement from one valley to another is called
a }\textit{\textcolor{black}{metastable transition}}\textcolor{black}{,
and the formula providing the precise asymptotics, as $\epsilon\rightarrow0$,
of the mean time taken to observe this transition is called the Eyring--Kramers
formula. Here, we emphasize that the Eyring--Kramers formulas for
the processes $\boldsymbol{y}_{\epsilon}(\cdot)$ and $\boldsymbol{x}_{\epsilon}(\cdot)$
have been established in \cite{BEGK1} and \cite{LeeSeo}, respectively. }

\textcolor{black}{}
\begin{figure}
\textcolor{black}{\includegraphics[scale=0.2]{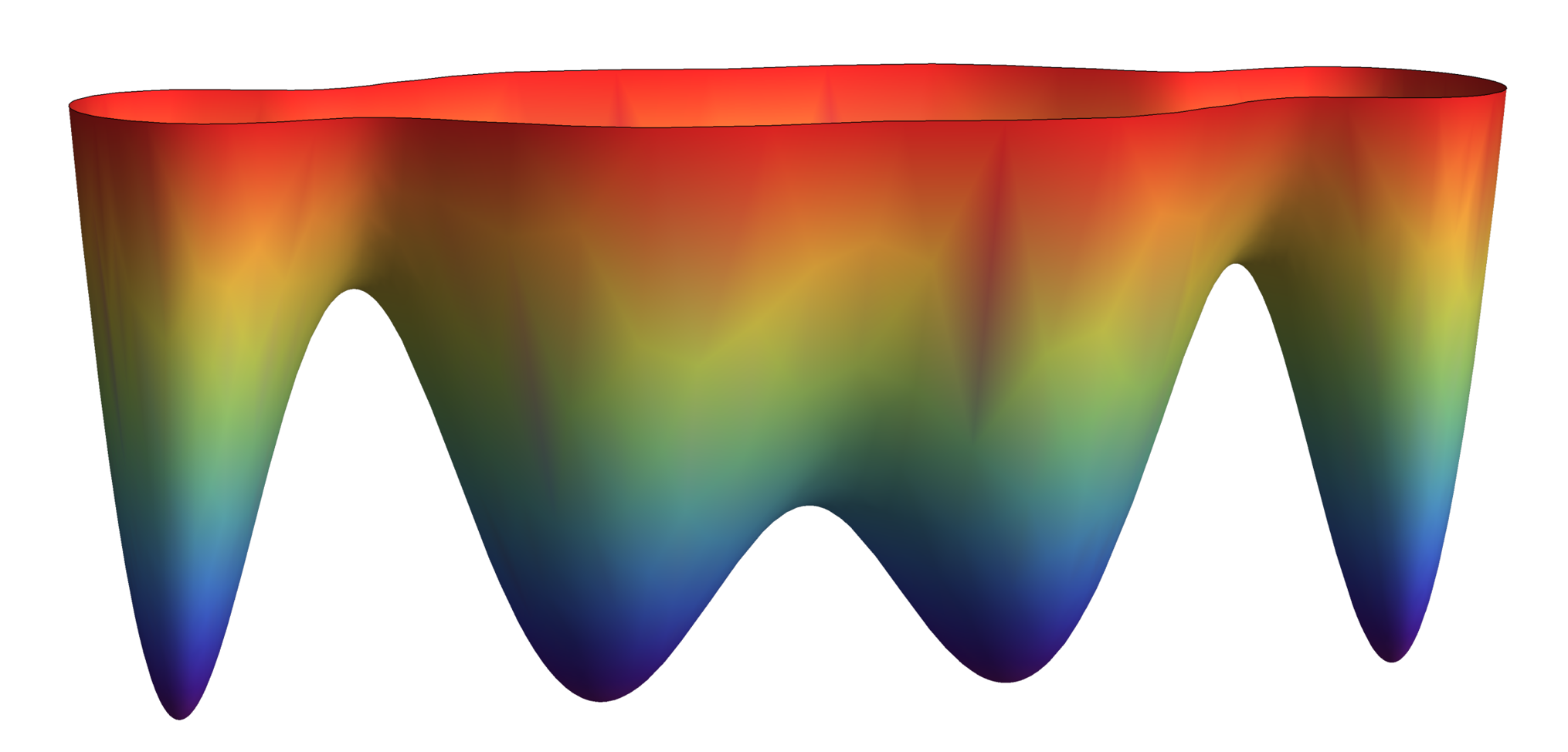}}

\textcolor{black}{\caption{\label{fig1}Example of potential $U$ with multiple global minima.}
}
\end{figure}
\textcolor{black}{{} }

\textcolor{black}{Our focus in this paper is not on such a single
transition but on the full description of successive transitions via
a suitable scaling limit. More precisely, from Figure}\textbf{\textcolor{black}{{}
}}\textcolor{black}{\ref{fig1}, we can expect that once the process
$\boldsymbol{y}_{\epsilon}(\cdot)$ makes a transition from one valley
to another, then the next transition to another valley will take place
after another exponentially long time. Hence, to comprehensively describe
the metastable behavior, it is natural to prove that these successive
metastable transitions converge in some sense to a continuous-time
Markov process whose state space consists of the valleys of $U$.
This proof of course requires highly accurate knowledge regarding
the transition time in the level of the Eyring--Kramers formula,
thereby providing a more detailed description of the metastable behavior.
This proof has been presented for the reversible process $\boldsymbol{y}_{\epsilon}(\cdot)$
in \cite{RS} based on the partial differential equation (PDE) approach}.
In this article, we extend the PDE approach to the non-reversible
setting in a robust manner and apply this method to the process $\boldsymbol{x}_{\epsilon}(\cdot)$.
We note that the method of \cite{RS} relies on analysis of solutions
of certain form of Poisson equations. It has been observed in \cite{LMS2,LMS3}
that considering resolvent equation, instead of Poisson equation,
simplifies several argument and provide more robust methodology in
the analysis of metastability. Hence, we will rely on this resolvent
approach to analyze the metastability of the process $\boldsymbol{x}_{\epsilon}(\cdot)$
in the current article.

Now, we review existing studies on the Markov chain description of
metastable behavior. In \cite{BL1,BL2}, a ro\textcolor{black}{bust
methodology based on potential theory has been introduced for the
case in which the underlying dynamic is a Markov process on a }\textit{\textcolor{black}{discrete}}\textcolor{black}{{}
set. This method has been applied to many models such as the zero-range
process \cite{BL3,Lan3,Seo}, the inclusion process \cite{BDG,Kim,KS},
the discrete version of the overdamped Langevin dynamics \cite{LMT,LS1},
and the ferromagnetic systems \cite{LL,LS2}. We refer to the references
in the above-mentioned articles for numerous other applications. On
the other hand, for metastable }\textit{\textcolor{black}{diffusion}}\textcolor{black}{{}
processes, a different methodology known as the PDE approach based
on the analysis of a certain Poisson equation has been introduced
in \cite{LS3,RS}. In \cite{LS3}, a general non-reversible metastable
diffusion process (cf. the process $\boldsymbol{z}_{\epsilon}(\cdot)$
below) on a}\textcolor{black}{\emph{ one-dimensional torus}}\textcolor{black}{{}
has been analyzed based on the explicit form of the solution to the
Poisson equation. In \cite{RS}, a general methodology to deal with
the solution to the corresponding Poisson equation when the underlying
dynamics is reversible has been developed and successfully applied
to the process $\boldsymbol{y}_{\epsilon}(\cdot)$}. It is observed
in \cite{LMS3} that replacing the corresponding Poisson equation
with \textbf{\emph{resolvent equation}} provides even more robust
and convenient methodology which in some sense provides a necessary
and sufficient condition for the Markov chain description of metastable
behavior. This method has been applied to a critical reversible zero-range
process \cite{LMS2} to which the method of \cite{BL1,BL2} is not
applicable because the metastable valley is too large. Our purpose
is to provide the first application of this method to a non-reversible
model.

\subsubsection*{Remarks on the process $\boldsymbol{x}_{\epsilon}(\cdot)$}

We conclude the introduction by explaining the importance of the process
$\boldsymbol{x}_{\epsilon}(\cdot)$ in the study of metastability.
To this end, let us consider a diffusion process given by an SDE in
$\mathbb{R}^{d}$ of the form
\begin{equation}
d\boldsymbol{z}_{\epsilon}(t)\,=\,-\boldsymbol{b}(\boldsymbol{z}_{\epsilon}(t))\,dt+\sqrt{2\epsilon}\,d\bm{w}_{t}\label{e_sdez}
\end{equation}
for some vector field $\boldsymbol{b}:\mathbb{R}^{d}\rightarrow\mathbb{R}^{d}$.
Suppose that the dynamical system in $\mathbb{R}^{d}$ given by the
ODE
\begin{equation}
d\boldsymbol{z}(t)\,=\,-\boldsymbol{b}(\boldsymbol{z}(t))\,dt\label{e_odez}
\end{equation}
has several stable equilibria so that the process $\boldsymbol{z}_{\epsilon}(\cdot)$
exhibits metastability. For this model, Freidlin and Wentzell \cite{FW}
established the large-deviation-type analysis of the metastable behavior.
However, rigorous accurate quantitative analysis such as that based
on the Eyring--Kramers formula or the Markov chain description is
unknown for this general model and remains as a primary open question
in this field. We refer to \cite{BR} for the Eyring--Kramers formula
for $\boldsymbol{z}_{\epsilon}(\cdot)$ under a special set of assumptions.

The difficulty in the rigorous analysis of the process $\boldsymbol{z}_{\epsilon}(\cdot)$
is due to two factors: the non-reversibility and the lack of an explicit
formula for the invariant measure. In Theorem \ref{t22} below, we
prove that the process $\boldsymbol{z}_{\epsilon}(\cdot)$ defined
in \eqref{e_sdez} has a Gibbs invariant measure \eqref{e_inv} if
and only if $\boldsymbol{b}=\nabla U+\boldsymbol{\ell}$ for some
$\boldsymbol{\ell}$ such that $\nabla U\cdot\boldsymbol{\ell}\equiv0$
and $\nabla\cdot\boldsymbol{\ell}\equiv0$; hence, this is the model
considered in this article. Thus, we completely overcome the difficulty
arising from the non-reversibility in the study of the process $\boldsymbol{z}_{\epsilon}(\cdot)$
in this article as wel\textcolor{black}{l as in \cite[Theorem 3.5]{LeeSeo}.
The} problem arising from the lack of an understanding of the invariant
measure of $\boldsymbol{z}_{\epsilon}(\cdot)$ is not addressed in
our studies, as the model considered has an explicit Gibbs invariant
measure; this problem should be investigated in future research.

\textcolor{black}{We finally remark that an important model which
is not discussed in this introduction is the underdamped Langevin
dynamics. This dynamics is non-reversible and has the Gibbs measure
as the invariant measure so that it seems at first glance that this
model falls into our framework. However, the main challenge in this
model is the fact that the diffusion coefficient is degenerate. Accordingly,
rigorous quantitative study for this model is barely known (cf. \cite{new_LRR})
and is an important future research problem. }

\section{\label{sec2}Model}

In this section, we review several basic features of the diffusion
process $\boldsymbol{x}_{\epsilon}(\cdot)$.

\subsubsection*{Potential function $U$ and vector field $\boldsymbol{\ell}$}

Initially, we present several assumptions on the potential function
$U$ and vector field $\boldsymbol{\ell}$ appearing in SDE \eqref{e_sdex}.
First, we assume that the potential function $U\in C^{2}(\mathbb{R}^{d})$
satisfies the following well-known growth conditions:
\begin{equation}
\begin{aligned} & \lim_{n\to\infty}\,\inf_{|\bm{x}|\geq n}\frac{U(\bm{x})}{|\bm{x}|}\,=\,\infty\;,\;\;\;\lim_{|\bm{x}|\to\infty}\frac{\bm{x}}{|\bm{x}|}\cdot\nabla U(\bm{x})\,=\,\infty\;,\text{ and }\\
 & \qquad\lim_{|\bm{x}|\to\infty}\{\,|\nabla U(\bm{x})|-2\Delta U(\bm{x})\,\}=\,\infty\;.
\end{aligned}
\label{e_conU}
\end{equation}
 A consequence (cf. \cite{BEGK1}) of these conditions is a bound
of the form
\begin{equation}
\int_{\{\bm{x}:U(\bm{x})\geq a\}}\,e^{-U(\bm{x})/\epsilon}\,d\bm{x}\,\leq\,C_{a}\,e^{-a/\epsilon}\text{\;\;\;for all }a\in\mathbb{R}\;,\label{e_tight}
\end{equation}
where $C_{a}$ is a constant depending only on $a$. We can also deduce
from the last bound along with the first condition of \eqref{e_conU}
that $Z_{\epsilon}<\infty$ where $Z_{\epsilon}$ is the partition
function defined in \eqref{e_partion}. Finally, we also assume that
$U$ is a Morse function, i.e., all the critical points of $U$ are
non-degenerate.

Now, we assume that $\boldsymbol{\ell}\in C^{1}(\mathbb{R}^{d},\,\mathbb{R}^{d})$
is a vector field that is orthogonal to the gradient field $\nabla U$
in the sense that
\begin{equation}
\nabla U(\boldsymbol{x})\cdot\boldsymbol{\ell}(\boldsymbol{x})\,=\,0\;\text{\;\;for all }\boldsymbol{x}\in\mathbb{R}^{d}\;.\label{e_conell1}
\end{equation}
With this assumption along with \eqref{e_conU}, we can prove that
the process $\boldsymbol{x}_{\epsilon}(\cdot)$ defined by SDE \eqref{e_sdex}
is non-explosive and positive recurrent (cf. \cite[Theorem 2.2]{LeeSeo}).
We also remark that \eqref{e_conell1} is equivalent to saying that
$U$ is the quasi-potential of the process $\boldsymbol{x}_{\epsilon}(\cdot)$
(see \cite[Theorem 3.3.1]{FW}) and hence is a natural assumption.

\subsubsection*{Dynamical system $\boldsymbol{x}(\cdot)$}

The process $\boldsymbol{x}_{\epsilon}(\cdot)$ exhibits metastability
when the following zero-noise (i.e., $\epsilon=0$) dynamics $\bm{x}(\cdot)$
has multiple stable equilibria:
\begin{equation}
d\bm{x}(t)\,=\,-(\nabla U+\boldsymbol{\ell})(\bm{x}(t))\,dt\;.\label{e_odex}
\end{equation}
It is also important to know the saddle points of this dynamical system
since they are crucial in the investigation of metastable transitions.
The next theorem analyzes the equilibria of the dynamical system \eqref{e_odex}
in terms of the critical points of $U$.
\begin{thm}
\label{t21}Under the assumptions on $U$ and $\boldsymbol{\ell}$
given above, a point $\boldsymbol{x}\in\mathbb{R}^{d}$ is an equilibrium
of the dynamical system \eqref{e_odex} if and only if $\boldsymbol{x}$
is a critical point of $U$. Moreover, an equilibrium $\boldsymbol{m}\in\mathbb{R}^{d}$
of \eqref{e_odex} is stable if and only if $\boldsymbol{m}$ is a
local minimum of $U$.
\end{thm}

\begin{proof}
See \cite[Theorem 2.1]{LeeSeo}.
\end{proof}
As a consequence of the previous theorem, we can observe that the
process $\boldsymbol{x}_{\epsilon}(\cdot)$ exhibits metastability
when $U$ has multiple local minima.

\subsubsection*{Divergence-free condition and Gibbs invariant measure}

Finally, we impose the incompressibility condition for vector field
$\boldsymbol{\ell}$:
\begin{equation}
(\nabla\cdot\boldsymbol{\ell})(\boldsymbol{x})\,=\,0\;\;\;\text{for all }\boldsymbol{x}\in\mathbb{R}^{d}\;.\label{e_conell2}
\end{equation}
This assumption is introduced to guarantee that the process $\boldsymbol{x}_{\epsilon}(\cdot)$
has a Gibbs invariant measure in the following sense. Recall that
the Gibbs measure $\mu_{\epsilon}(\cdot)$ on $\mathbb{R}^{d}$ is
defined by \eqref{e_inv}.
\begin{thm}
\label{t22}If $\boldsymbol{\ell}$ satisfies conditions \eqref{e_conell1}
and \eqref{e_conell2}, then the Gibbs measure $\mu_{\epsilon}(\cdot)$
is the unique invariant measure for the diffusion process $\boldsymbol{x}_{\epsilon}(\cdot)$.
Conversely, if the Gibbs measure $\mu_{\epsilon}(\cdot)$ is the invariant
measure for the diffusion process $\boldsymbol{z}_{\epsilon}(\cdot)$
defined in \eqref{e_sdez} with $\boldsymbol{b}\in C^{1}(\mathbb{R}^{d},\,\mathbb{R}^{d})$
for all $\epsilon>0$, then, the vector field can be written as $\boldsymbol{b}=\nabla U+\boldsymbol{\ell}$,
where the function $U$ and vector field $\boldsymbol{\ell}$ satisfy
\eqref{e_conell1} and \eqref{e_conell2}.
\end{thm}

\begin{proof}
See \cite[Theorem 2.3]{LeeSeo}
\end{proof}
One may ask whether it is easy to find a vector field $\boldsymbol{\ell}$
satisfying \eqref{e_conell1} and \eqref{e_conell2}. For a given
potential $U$, there is a well-known, simple procedure that provides
a sufficiently wide variety of selections for $\boldsymbol{\ell}$.
To explain this, let $J:\mathbb{R}\rightarrow\mathcal{M}_{d\times d}(\mathbb{R})$
be a smooth map from $\mathbb{R}$ to the space of $d\times d$ matrices
such that $J(a)$ is skew-symmetric for all $a\in\mathbb{R}$. Then,
the vector field of the form $\boldsymbol{\ell}(\boldsymbol{x})=J(U(\boldsymbol{x}))\nabla U(\boldsymbol{x})$
satisfies \eqref{e_conell1} and \eqref{e_conell2}, as observed in
\cite[Section 1]{LL}.

\subsubsection*{Generator and its adjoint}

We conclude this section by introducing the generator of the process
$\boldsymbol{x}_{\epsilon}(\cdot)$ and its adjoint. The generator
$\mathscr{L}_{\epsilon}$ associated with the process $\boldsymbol{x}_{\epsilon}(\cdot)$
acts on $f\in C^{2}(\mathbb{R}^{d})$ such that
\[
(\mathscr{L}_{\epsilon}f)(\boldsymbol{x})\,=\,-(\nabla U(\bm{x})+\boldsymbol{\ell}(\bm{x}))\cdot\nabla f(\bm{x})+\epsilon\,\Delta f(\bm{x})\;.
\]
By \eqref{e_conell1} and \eqref{e_conell2}, we can rewrite this
generator in divergence form as
\begin{equation}
(\mathscr{L}_{\epsilon}f)(\boldsymbol{x})\,=\,\epsilon\,e^{U(\bm{x})/\epsilon\,}\nabla\cdot\,\Big[\,e^{-U(\bm{x})/\epsilon}\,\Big(\,\nabla f(\bm{x})-\frac{1}{\epsilon}\,f(\bm{x})\,\boldsymbol{\ell}(\bm{x})\,\Big)\,\Big]\;.\label{e_gen}
\end{equation}
The adjoint operator $\mathscr{L}_{\epsilon}^{*}$ of $\mathscr{L}_{\epsilon}$
with respect to the measure $d\mu_{\epsilon}$ can be written as
\begin{align}
(\mathscr{L}_{\epsilon}^{*}f)(\bm{x})\, & =\,\epsilon\,e^{U(\bm{x})/\epsilon}\,\nabla\cdot\,\Big[\,e^{-U(\bm{x})/\epsilon}\,\Big(\,\nabla f(\bm{x})+\frac{1}{\epsilon}\,f(\bm{x})\,\boldsymbol{\ell}(\bm{x})\,\Big)\,\Big]\;.\label{e_genadj}
\end{align}

\section{\label{sec3}Main Result}

In this section, we explain our main result regarding the Markov chain
description of the metastable behavior of the process $\boldsymbol{x}_{\epsilon}(\cdot)$
when $U$ has several local minima.

\subsection{\label{sec31}Landscape of $U$ and invariant measure}

We first analyze the landscape of $U$. We refer to Figure \ref{fig31}
for an illustration of the notations introduced in this subsection.

\begin{figure}
\includegraphics[scale=0.21]{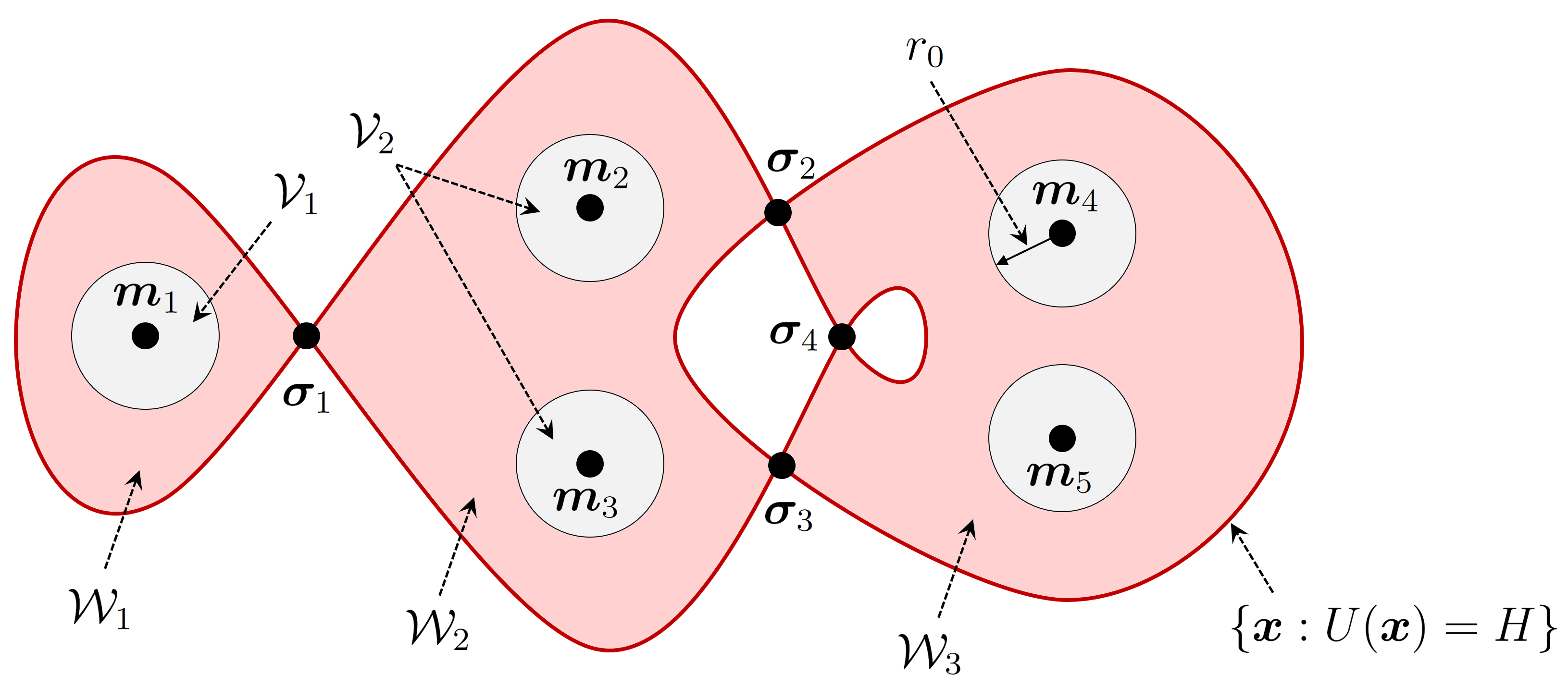}\caption{\label{fig31}Example of landscape of $U$. In this example, we have
$\Sigma=\{\boldsymbol{\sigma}_{1},\,\boldsymbol{\sigma}_{2},\,\boldsymbol{\sigma}_{3},\,\boldsymbol{\sigma}_{4}\}$
and the set $\{\boldsymbol{x}:U(\boldsymbol{x})<H\}$ consists of
three components $\mathcal{W}_{1},\,\mathcal{W}_{2},\,\mathcal{W}_{3}$.
Hence, $S=\{1,\,2,\,3\}$. We have $\Sigma_{1,\,2}=\{\boldsymbol{\sigma}_{1}\}$,
$\Sigma_{2,\,3}=\{\boldsymbol{\sigma}_{2},\,\boldsymbol{\sigma}_{3}\}$,
and $\Sigma_{1,\,3}=\emptyset$. Therefore, $\Sigma^{*}=\{\boldsymbol{\sigma}_{1},\,\boldsymbol{\sigma}_{2},\,\boldsymbol{\sigma}_{3}\}\subsetneq\Sigma$.
Suppose that $h_{1}=h_{2}=h<h_{3}$. Then, we have $S_{\star}=\{1,\,2\}$.
By assuming that $U(\boldsymbol{m}_{2})=U(\boldsymbol{m}_{3})=h$,
two metastable valleys are defined by $\mathcal{V}_{1}=\overline{\mathcal{D}_{r_{0}}(\boldsymbol{m}_{1})}$
and $\mathcal{V}_{2}=\overline{\mathcal{D}_{r_{0}}(\boldsymbol{m}_{2})}\cup\overline{\mathcal{D}_{r_{0}}(\boldsymbol{m}_{3})}$.
Metastable valley is not defined for the shallow well $\mathcal{W}_{3}$. }
\end{figure}

For a concrete description, we fix a level $H$ and define $\Sigma=\Sigma_{H}$
as the set of saddle points of level $H$:
\[
\Sigma\,:=\,\{\bm{\sigma}:U(\bm{\sigma})=H\text{ and }\boldsymbol{\sigma}\text{ is a }\text{saddle point of }U\}\ .
\]
By selecting $H$ appropriately, we shall assume that $\Sigma$ is
a non-empty set. We now define
\[
\mathcal{H}\,:=\,\{\bm{x}\in\mathbb{R}^{d}:U(\bm{x})<H\}\;,
\]
and denote by $\mathcal{W}_{1},\,\dots,\,\mathcal{W}_{K}$ the connected
components of the set $\mathcal{H}$. These sets are called (metastable)
wells for the potential function $U$ corresponding to the level $H$.
We focus on the transition of the process $\boldsymbol{x}_{\epsilon}(\cdot)$
among these wells. Various selections of $H$ are possible; however,
we focus on one fixed level to get a concrete result. The last paragraph
of the current section explain how we can select various $H$ to get
a variety of results that provide a full description of the metastable
behavior.

If $K=1$, there is no interesting metastable behavior at level $H$,
and we must take a smaller level to observe the metastable behavior.
Therefore, we assume that $K\ge2$. Now, we shall assume that the
closure $\mathcal{\overline{H}}$ of $\mathcal{H}$ is a connected
set. Otherwise, our analysis can be applied to each connected component
of $\mathcal{\overline{H}}$, and this general situation is explained
later. See the discussion after Theorem \ref{t_main}.

Write $S=\{1,\cdots,K\}$. For $i,\,j\in S$,\footnote{In this article, writing ``$a,\,b$'' implies that $a$ and $b$
are distinct.} we write
\[
\Sigma_{i,\,j}\,=\,\overline{\mathcal{W}}_{i}\cap\overline{\mathcal{W}}_{j}\ ,
\]
which denotes the set of saddle points between $\mathcal{W}_{i}$
and $\mathcal{W}_{j}$ of level $H$. Note that this set can be empty.
Now, assume further that $\Sigma_{i,\,j}\cap\Sigma_{k,\,l}=\emptyset$
unless $\{i,\,j\}=\{k,\,l\}$; hence, there is no saddle point connecting
three or more wells simultaneously. Write
\begin{equation}
\Sigma^{*}\,=\,\bigcup_{i,\,j\in S}\,\Sigma_{i,\,j}\;.\label{e_sig*}
\end{equation}
Then, we have $\Sigma^{*}\subseteq\Sigma$, and the equality may not
hold (cf. Figure \ref{fig31}). By the Morse lemma,\textbf{ }for each
$\boldsymbol{\sigma}\in\Sigma^{*}$, the Hessian $(\nabla^{2}U)(\boldsymbol{\sigma})$
has only one negative eigenvalue and $(d-1)$ positive eigenvalues,
as we have assumed that $U$ is a Morse function. We remark that this
may not be true for $\boldsymbol{\sigma}\in\Sigma\setminus\Sigma^{*}$.

\subsubsection*{Metastable valleys}

Now, we define the metastable valleys. We fix $i\in S$ and denote
by $h_{i}$ the minimum value of the potential $U$ on the well $\mathcal{W}_{i}$,
i.e.,
\begin{equation}
h_{i}\,:=\,\min\{U(\boldsymbol{x}):\boldsymbol{x}\in\mathcal{W}_{i}\}\;.\label{e_hi}
\end{equation}
Define $\mathcal{M}_{i}$ as the set of the deepest minima of $U$
on $\mathcal{W}_{i}$:
\[
\mathcal{M}_{i}\,:=\,\{\bm{m}\in\mathcal{W}_{i}:U(\bm{m})=h_{i}\}\;.
\]
Then, we can regard $H-h_{i}$ as the depth of the well $\mathcal{W}_{i}$.
We write the ball in $\mathbb{R}^{d}$ centered at $\boldsymbol{x}$
with radius $r$ as
\[
\mathcal{D}_{r}(\bm{x})\,:=\,\{\boldsymbol{y}\in\mathbb{R}^{d}:|\boldsymbol{y}-\boldsymbol{x}|<r\}\;.
\]
We take $r_{0}>0$ to be sufficiently small so that, for all $i\in S$
and for all $\boldsymbol{m}\in\mathcal{M}_{i}$,
\begin{equation}
\mathcal{D}_{2r_{\text{0}}}(\bm{m})\subset\mathcal{W}_{i}\text{ and \ensuremath{\overline{\mathcal{D}_{2r_{\text{0}}}(\bm{m})}\setminus\{\boldsymbol{m}\}} does not contain a critical point of }U.\label{conr0}
\end{equation}

Finally, the \textit{metastable valley} corresponding to the well
$\mathcal{W}_{i}$ is defined as
\begin{equation}
\mathcal{V}_{i}\,:=\,\bigcup_{\bm{m}\in\mathcal{M}_{i}}\overline{\mathcal{D}_{r_{\text{0 }}}(\bm{m})}\;,\label{e_Vi}
\end{equation}
where $\overline{\mathcal{D}_{r_{\text{0 }}}(\bm{m})}=\{\boldsymbol{y}\in\mathbb{R}^{d}:|\boldsymbol{y}-\boldsymbol{x}|\le r_{0}\}$
denotes the closed ball. Our primary focus is the inter-valley dynamics
among these sets $\mathcal{V}_{i}$.

\subsubsection*{Deepest valleys}

We now characterize the deepest valleys of $U$, which will be the
state space of the limiting Markov chain describing the metastable
behavior. Recall $h_{i}$ from \eqref{e_hi} and define
\[
h\,:=\,\min_{i\in S}h_{i}\;\;\;\text{and\;\;\;}S_{\star}\,:=\,\{i\in S:h_{i}=h\}\;,
\]
so that $\{\mathcal{W}_{i}:i\in S_{\star}\}$ denotes the collection
of the deepest wells. We assume that \textbf{$|S_{\star}|\ge2$} since
the Markov chain description is trivial when $|S_{\star}|=1$.\textbf{
}Let
\[
\mathcal{M}_{\star}:=\,\bigcup_{i\in S_{\star}}\mathcal{M}_{i}\,=\,\{\boldsymbol{x}\in\mathbb{R}^{d}:U(\boldsymbol{x})=h\}\;,
\]
so that the set $\mathcal{M}_{\star}$ denotes the set of global minima
of $U$. Write $\mathcal{V}_{\star}=\bigcup_{i\in S_{\star}}\mathcal{V}_{i}$
so that $\mathcal{V}_{\star}$ denotes the set of deepest valleys.
Finally, we write $\Delta=\mathbb{R}^{d}\setminus\mathcal{V}_{\star}.$

\subsubsection*{Invariant measure}

With the construction of the metastable valleys, we can conclude that
the invariant measure $\mu_{\epsilon}(\cdot)$ is concentrated on
the set $\mathcal{V}_{\star}$. Moreover, we can compute the precise
asymptotics for $\mu_{\epsilon}(\mathcal{V}_{i})$ for each $i\in S_{\star}.$
To this end, we first introduce several notations.
\begin{notation}
\label{not31}For each $\boldsymbol{x}\in\mathbb{R}^{d}$, we write
$\mathbb{H}^{\boldsymbol{x}}=(\nabla^{2}U)(\boldsymbol{x})$ as the
Hessian of $U$ at $\boldsymbol{x}$ and $\mathbb{L}^{\boldsymbol{x}}=D\boldsymbol{\ell}(\bm{x})$
as the Jacobian of $\boldsymbol{\ell}$ at $\boldsymbol{x}$.
\end{notation}

For each $i\in S$, we define
\[
\nu_{i}\,:=\,\sum_{\bm{m}\in\mathcal{M}_{i}}\frac{1}{\sqrt{\det\mathbb{H}^{\boldsymbol{m}}}}\ ,
\]
and write $\nu_{\star}=\sum_{i\in S_{\star}}\nu_{i}$. For a sequence
$(a_{\epsilon})_{\epsilon>0}$ of real numbers, we write $a_{\epsilon}=o_{\epsilon}(1)$
if $\lim_{\epsilon\to0}a_{\epsilon}=0$.The following asymptotics
are useful in our discussion.
\begin{prop}
\label{p32}We have
\begin{align}
Z_{\epsilon} & \,=\,[\,1+o_{\epsilon}(1)\,]\,(2\pi\epsilon)^{d/2}\,e^{-h/\epsilon}\,\nu_{\star}\;,\label{e_Zeps}\\
\mu_{\epsilon}(\mathcal{V}_{i}) & \,=\,[\,1+o_{\epsilon}(1)\,]\,\frac{\nu_{i}}{\nu_{\star}}\ \text{\;\;;\;}i\ensuremath{\in S_{\star}}\;\;\;\;\text{and\;\;\;\;}\mu_{\epsilon}(\Delta)\,=\,o_{\epsilon}(1)\;.\nonumber
\end{align}
\end{prop}

\begin{proof}
The proof is a consequence of an elementary computation based on the
Laplace asymptotics. For further detail, we refer to \cite[Proposition 2.2]{RS}.
\end{proof}

\subsubsection*{Eyring--Kramers constants}

For $\boldsymbol{\sigma}\in\Sigma^{*}$, we previously mentioned that
the Hessian $\mathbb{H}^{\boldsymbol{\sigma}}$ has only one negative
eigenvalue by the Morse lemma. We can further show that the matrix
$\mathbb{H}^{\boldsymbol{\sigma}}+\mathbb{L}^{\boldsymbol{\sigma}}$
also has only one negative eigenvalue by an elementary computation
carried out in the proof of \cite[Lemma 3.3]{LeeSeo}. Denote by $-\mu^{\boldsymbol{\sigma}}$
the unique negative eigenvalue of $\mathbb{H}^{\boldsymbol{\sigma}}+\mathbb{L}^{\boldsymbol{\sigma}}$.
Then, the \textit{Eyring--Kramers constant} at $\boldsymbol{\sigma}\in\Sigma^{*}$
is defined by
\begin{equation}
\omega^{\boldsymbol{\sigma}}\,:=\,\frac{\mu^{\bm{\sigma}}}{2\pi\sqrt{-\det\mathbb{H}^{\boldsymbol{\sigma}}}}\;.\label{e_EKconst}
\end{equation}
For $i,\,j\in S$, we define
\[
\omega_{i,\,j}\,:=\,\sum_{\bm{\sigma}\in\Sigma_{i,\,j}}\omega^{\boldsymbol{\sigma}}\;\;\;\text{and\;\;\;}\omega_{i}\,:=\,\sum_{j\in S}\omega_{i,\,j}\;,
\]
where we set $\omega_{i,\,i}=0$ for $i\in S$ for convenience of
notation. Note that the connectedness of $\overline{\mathcal{H}}$
implies that $\omega_{i}>0$ for all $i\in S$.

\subsection{\label{sec32}Two Markov chains}

Now, we construct two continuous-time Markov chains: $(\mathbf{x}(t))_{t\ge0}$
and $(\mathbf{y}(t))_{t\ge0}$. The Markov chain $\mathbf{y}(\cdot)$
describes the limiting metastable behavior of the diffusion process
$\boldsymbol{x}_{\epsilon}(\cdot)$. The auxiliary Markov chain $\mathbf{x}(\cdot)$
is used in the construction of this limiting chain $\mathbf{y}(\cdot)$;
moreover, it plays a crucial role in the proof. \textcolor{black}{We
refer to Remark \ref{rem33} for the meaning of these Markov chains. }

\textcolor{black}{The construction of the limiting chain $\mathbf{y}(\cdot)$
is simple when all the wells have the same depth, i.e., $S_{\star}=S$
(cf. Remark \ref{rem33}).}\textcolor{red}{{} }However, if $S_{\star}\subsetneq S$,
the behavior of the process $\boldsymbol{x}_{\epsilon}(\cdot)$ on
each shallow valley $\mathcal{V}_{i}$, $i\in S\setminus S_{\star}$,
should be properly reflected in the construction; hence, the definition
of $\mathbf{y}(\cdot)$ becomes more complex and should be done via
the auxiliary chain $\mathbf{x}(\cdot)$ defined from now on.

\subsubsection*{Auxiliary Markov chain $\mathbf{x}(\cdot)$ on $S$}

We define a probability measure $m(\cdot)$ on $S$ by
\[
m(i)\,:=\,\omega_{i}/\sum_{j\in S}\omega_{j}\;\;\;;\;i\in S\;.
\]
Let $(\mathbf{x}(t))_{t\geq0}$ be the continuous-time Markov chain
on $S$ whose jump rate from $i\in S$ to $j\in S$ is given by $r_{\mathbf{x}}(i,\,j)=\omega_{i,\,j}/m(i)$.
It is clear that the invariant measure for the Markov chain $\mathbf{x}(\cdot)$
is $m(\cdot)$, and moreover the process $\mathbf{x}(\cdot)$ is reversible
with respect to $m(\cdot)$. We now introduce several potential theoretic
notions regarding the process $\mathbf{x}(\cdot)$. These notions
are used in the definition of the limiting Markov chain $\mathbf{y}(\cdot)$.

Denote by $L_{\mathbf{x}}$ the generator associated with the Markov
chain $\mathbf{x}(\cdot)$ acting on $\mathbf{f}:S\rightarrow\mathbb{R}$
such that
\[
(L_{\mathbf{x}}\mathbf{f})(i)\,=\,\sum_{j\in S}r_{\mathbf{x}}(i,\,j)\,[\,\mathbf{f}(j)-\mathbf{f}(i)\,]\;\;\ ;\ i\in S\;.
\]
Denote by $\mathbf{P}_{i}$ the law of process $\mathbf{x}(\cdot)$
starting at $i\in S$. For two disjoint non-empty subsets $A,\,B$
of $S$, the equilibrium potential between $A$ and $B$ with respect
to the process $\mathbf{x}(\cdot)$ is a function $\mathbf{h}_{A,\,B}:S\rightarrow\mathbb{R}$
defined by
\[
\mathbf{h}_{A,\,B}(i)\,:=\,\mathbf{P}_{i}\,[\,\tau_{A}<\tau_{B}\,]\;\;\;;\;i\in S\;,
\]
where $\tau_{A},\,A\subset S$, denotes the hitting time of the set
\textcolor{black}{$A$, i.e., $\tau_{A}=\inf\{t\ge0\,:\,{\bf x}(t)\in A\}$.}
Define a bi-linear form $D_{\mathbf{x}}(\cdot,\cdot)$ by, for all
$\mathbf{f},\,\mathbf{g}:S\rightarrow\mathbb{R}$,
\begin{equation}
D_{\mathbf{x}}(\mathbf{f},\,\mathbf{g})\,:=\,\sum_{i\in S}\mu(i)\,\mathbf{f}(i)\,[\,-(L_{\mathbf{x}}\mathbf{g})(i)\,]=\frac{1}{2}\,\sum_{i,\,j\in S}\omega_{i,\,j}\,[\,\mathbf{f}(i)-\mathbf{f}(j)\,]\,[\,\mathbf{g}(i)-\mathbf{g}(j)\,]\;.\label{e_Dx}
\end{equation}
Note that $D_{\mathbf{x}}(\mathbf{f},\,\mathbf{f})$ represents the
Dirichlet form associated with the Markov chain $\mathbf{x}(\cdot)$.
Finally, the capacity between two disjoint non-empty subsets $A$
and $B$ of $S$ with respect to the process $\mathbf{x}(\cdot)$
is defined by
\begin{equation}
\textrm{cap}_{\mathbf{x}}(A,\,B)\,:=\,D_{\mathbf{x}}(\mathbf{h}_{A,\,B},\,\mathbf{h}_{A,\,B})\;.\label{e_capx}
\end{equation}

\subsubsection*{Limiting Markov chain $\mathbf{y}(\cdot)$ on $S_{\star}$}

Recall that we assumed $|S_{\star}|\geq2$. For $i,\,j\in S_{\star}$,
define
\[
\beta_{i,\,j}\,:=\,\frac{1}{2}[\textrm{cap}_{\mathbf{x}}(\{i\},\,S_{\star}\setminus\{i\})+\textrm{cap}_{\mathbf{x}}(\{j\},\,S_{\star}\setminus\{j\})-\textrm{cap}_{\mathbf{x}}(\{i,j\},\,S_{\star}\setminus\{i,j\})]\;.
\]
We set $\beta_{i,\,i}=0$, $i\in S_{\star}$, for convenience and
note that we have $\beta_{i,\,j}=\beta_{j,\,i}$ for all $i,\,j\in S_{\star}$.
Then, we define $(\mathbf{y}(t))_{t\ge0}$ as a continuous-time Markov
chain on $S_{\star}$ with jump rate $r_{\mathbf{y}}(i,\,j)$ from
$i\in S_{\star}$ to $j\in S_{\star}$ given by $r_{\mathbf{y}}(i,\,j)=\beta_{i,j}/\nu_{i}$.\textcolor{black}{{}
The process $\mathbf{y}(\cdot)$ defined in this manner is indeed
the so-called trace process of $\mathbf{x}(\cdot)$ (cf. \cite[Appendix]{BL1})}
\begin{rem}[Comments on the processes ${\bf x}(\cdot)$ and ${\bf y}(\cdot)$]
\textcolor{black}{\label{rem33} The auxiliary process $\mathbf{x}(\cdot)$
represents the inter-valley dynamics of the process $\boldsymbol{x}_{\epsilon}(\cdot)$
by assuming that it spends the same time scale at all valleys (which
is not true in general). Since the process $\boldsymbol{x}_{\epsilon}(\cdot)$
spends a negligible time scale on shallow valleys, we can take the
suitable trace of the process ${\bf x}(\cdot)$ on the (indices corresponding
to) deepest valleys to get the correct process representing the inter-(deepest)
valley dynamics of the process $\boldsymbol{x}_{\epsilon}(\cdot)$.
This trace process is $\mathbf{y}(\cdot)$. }
\end{rem}

\subsection{\label{sec33}Markov chain description via convergence of order process}

Recall that $H-h$ represents the depth of the deepest wells. We can
expect from Eyring--Kramers formula for $\boldsymbol{x}_{\epsilon}(\cdot)$
obtained in \cite{LeeSeo} that the order of the time scale for a
metastable transition is
\[
\theta_{\epsilon}\,:=\,\exp\frac{H-h}{\epsilon}\;.
\]
Hence, we speed up the process $\boldsymbol{x}_{\epsilon}(\cdot)$
by a factor of $\theta_{\epsilon}$ and then observe the index of
the valley in which the speeded-up process is staying. To that end,
we write
\[
\widetilde{\boldsymbol{x}}_{\epsilon}(t)=\boldsymbol{x}_{\epsilon}(\theta_{\epsilon}t)\;\;\;\;;\;t\ge0
\]
the speeded-up process. In view of the fact that $\mu_{\epsilon}(\mathcal{V}_{\star})=1-o_{\epsilon}(1)$
(cf. Proposition \ref{p32}), this index belongs to the set $S_{\star}$
with dominating probability. We wish to prove that this index process
converges to the process $\mathbf{y}(\cdot)$ defined in the previous
subsection. The major technical issue in this heuristic explanation
is the fact that the speeded-up process $\widetilde{\boldsymbol{x}}_{\epsilon}(\cdot)$
may stay in the set $\Delta=\mathbb{R}\setminus\mathcal{V}_{\star}$
with small probability, and for this case, the index process is not
defined. Thus, to formulate this convergent result in a rigorous manner,
we recall the notion of the \textit{order process} introduced in \cite{BL1,BL2}.
To define the order process, define
\[
T_{\epsilon}(t)\,:=\,\int_{0}^{t}\mathbf{1}\{\widetilde{\boldsymbol{x}}_{\epsilon}(s)\in\mathcal{V}_{\star}\}\,ds\;\;\;\ ;\ t\geq0\ ,
\]
which measures the amount of time for which the speeded-up process
$\widetilde{\boldsymbol{x}}_{\epsilon}(\cdot)$ stayed in $\mathcal{V}_{\star}$
until time $t$. Then, define $S_{\epsilon}(t)$ as the generalized
inverse of the random increasing function $T_{\epsilon}(\cdot)$:
\begin{equation}
S_{\epsilon}(t)\,:=\,\sup\{s\geq0:T_{\epsilon}(s)\leq t\}\ \;\;;\ t\geq0\;.\label{e_St}
\end{equation}
Define the trace process $\boldsymbol{\xi}_{\epsilon}(\cdot)$ as
\[
\boldsymbol{\xi}_{\epsilon}(t)\,:=\,\widetilde{\boldsymbol{x}}_{\epsilon}(S_{\epsilon}(t))\;\;\ ;\ t\geq0\;.
\]
This process the one is obtained from the process $\widetilde{\boldsymbol{x}}_{\epsilon}(\cdot)=\boldsymbol{x}_{\epsilon}(\theta_{\epsilon}\cdot)$
by turning off the clock when the process $\widetilde{\boldsymbol{x}}_{\epsilon}(\cdot)$
does not belong to $\mathcal{V}_{\star}$. In other words, the trajectory
of $\boldsymbol{\xi}_{\epsilon}(\cdot)$ is obtained by removing the
excursions of $\widetilde{\boldsymbol{x}}_{\epsilon}(\cdot)$ at $\Delta$.
Hence, we have $\boldsymbol{\xi}_{\epsilon}(t)\in\mathcal{V}_{\star}$
for all $t\ge0$; furthermore, the process $\boldsymbol{\xi}_{\epsilon}(\cdot)$
is a Markov process (with jump) on $\mathcal{V}_{\star}$.

First, we show that the process $\boldsymbol{\xi}_{\epsilon}(\cdot)$
is a relevant approximation of the process $\widetilde{\boldsymbol{x}}_{\epsilon}(\cdot)$
in the sense that the excursion of $\widetilde{\boldsymbol{x}}_{\epsilon}(\cdot)$
at $\Delta$ is negligible. Denote by $\mathbb{P}_{\boldsymbol{x}}^{\epsilon}$
the law of the original process $\boldsymbol{x}_{\epsilon}(\cdot)$
starting from $\bm{x}\in\mathbb{R}^{d}$ and by $\mathbb{E}_{\boldsymbol{x}}^{\epsilon}$
the expectation with respect to it.
\begin{thm}
\label{t33_neg}For all $t\ge0$, it holds that
\[
\lim_{\epsilon\to0}\sup_{\boldsymbol{x}\in\mathcal{V}_{\star}}\mathbb{E}_{\boldsymbol{x}}^{\epsilon}\,\Big[\,\int_{0}^{t}\,\mathbf{1}_{\Delta}(\widetilde{\boldsymbol{x}}_{\epsilon}(s))\,ds\,\Big]\,=\,0\;.
\]
\end{thm}

The proof of this result is a direct consequence of the analysis of
resolvent equation explained in Section \ref{sec_res} and will be
explained therein.

By assuming this theorem, it now suffices to analyze the inter-valley
behavior of the trace process $\boldsymbol{\xi}_{\epsilon}(\cdot)$.
To this end, we define a projection $\Psi:\mathcal{V}_{\star}\rightarrow S_{\star}$
simply by
\begin{equation}
\Psi(\bm{x})\,=\,i\ \text{\;\;\;if}\ \bm{x}\in\mathcal{V}_{i}\;\;\;;\;i\in S_{\star}\;,\label{e_Psi}
\end{equation}
which maps a point belonging to a deepest valley to the index of that
valley. Finally, define a process on $S_{\star}$ as
\[
\mathbf{y}_{\epsilon}(t)\,:=\,\Psi(\boldsymbol{\xi}_{\epsilon}(t))\;\;\;\;;\;t\ge0\;,
\]
which represents the valley where the trace process $\boldsymbol{\xi}_{\epsilon}(t)$
is staying. This process $\mathbf{y}_{\epsilon}(\cdot)$ is called
the \textit{order process}. Denote by $\mathbf{Q}_{\pi_{\epsilon}}^{\epsilon}$
the law of the order process $\mathbf{y}_{\epsilon}(\cdot)$ when
the underlying process $\boldsymbol{x}_{\epsilon}(\cdot)$ starts
from a distribution $\pi_{\epsilon}$ on $\mathbb{R}^{d}$, and denote
by $\mathbf{Q}_{i}$ the law of the limiting Markov chain $\mathbf{y}(\cdot)$
starting from $i\in S_{\star}$. The following convergence theorem
is the main result of the current paper.
\begin{thm}
\label{t_main}For every $i\in S_{\star}$ and for any sequence of
Borel probability measures $(\pi_{\epsilon})_{\epsilon>0}$ concentrated
on $\mathcal{V}_{i}$, the law $\mathbf{Q}_{\pi_{\epsilon}}^{\epsilon}$
of the order process converges to $\mathbf{Q}_{i}$ as $\epsilon\rightarrow0$.
\end{thm}

The proof of the theorem based on the resolvent approach developed
in \cite{LMS3} is given in the next subsection. We remark that this
is a generalization of \cite[Theorem 2.3]{RS}, as the reversible
case is the special $\boldsymbol{\ell}=\boldsymbol{0}$ case of our
model. \textcolor{black}{Moreover, a careful reading of our arguments
reveals that, the speed of the convergence of the finite dimensional
marginals is given by}
\[
\mathbf{Q}_{\pi_{\epsilon}}^{\epsilon}[\mathbf{y}_{\epsilon}(t_{i})\in A_{i}\;\;\text{for }i=1,\,\dots,\,k]=\big(1+O\big(\epsilon^{1/2}\log\frac{1}{\epsilon}\big)\big)\,\mathbf{Q}_{i}[\mathbf{y}(t_{i})\in A_{i}\;\;\text{for }i=1,\,\dots,\,k]
\]
under the conditions of Theorem \ref{t_main}, where the error term
$O(\epsilon^{1/2}\log\frac{1}{\epsilon})$ is identical to the one
appeared in \cite[Theorems 3.1 and 3.2]{BEGK1} and depends on $t_{1},\,\dots,\,t_{k}$.

\subsubsection*{Discussion on general case}

Thus far, we have assumed that $\overline{\mathcal{H}}=\{\boldsymbol{x}\in\mathbb{R}^{d}:U(\boldsymbol{x})\le H\}$
is connected. However, our argument can be readily applied to the
general situation without this assumption as follows. If $\mathcal{\overline{H}}$
is not connected, we take a connected component $\mathcal{X}$ and
denote by $\mathcal{W}_{1},\,\dots,\,\mathcal{W}_{K}$ the connected
component of $\mathcal{H}$ contained in $\mathcal{X}$. Let $S=\{1,\,\dots,\,K\}$.
Then, we can define all the notations as before, and Theorem \ref{t_main}
holds unchanged. This can be readily verified by coupling with the
same dynamical systems reflected at the boundary of the connected
component of the domain $\{\boldsymbol{x}\in\mathbb{R}^{d}:U(\boldsymbol{x})<H+a\}$
for small enough $a$ containing $\mathcal{X}$. This will be more
precisely explained in \cite{LLS} at which all the inter-valley (not
restricted to the deepest valleys) dynamics are completely analyzed.
Therefore, we can vary $H$ to get different convergence results,
and an example is given in Figure \ref{fig2}.

\begin{figure}
\includegraphics[scale=0.21]{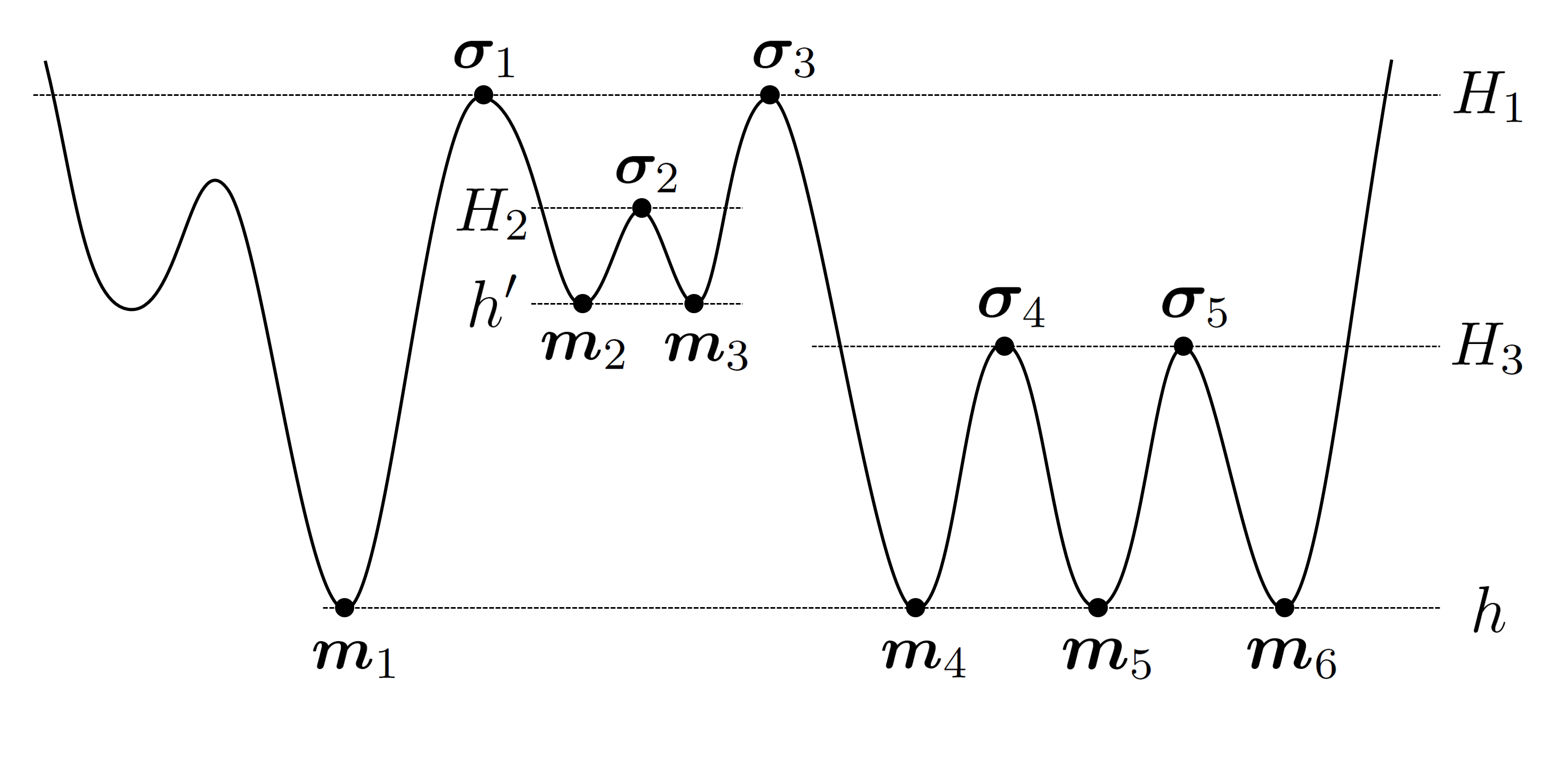}\caption{\label{fig2}In this example of $U$, we have three possible choices
of $H$: $H_{1}$, $H_{2}$, and $H_{3}$. By selecting $H=H_{1}$,
we analyze the transitions between two deepest valleys $\mathcal{D}_{r_{0}}(\boldsymbol{m}_{1})$
and $\mathcal{D}_{r_{0}}(\boldsymbol{m}_{4})\cup\mathcal{D}_{r_{0}}(\boldsymbol{m}_{5})\cup\mathcal{D}_{r_{0}}(\boldsymbol{m}_{6})$.
The time scale for these transitions is $e^{(H_{1}-h)/\epsilon}$,
and $\Sigma^{*}$ with this choice of $H$ is $\{\boldsymbol{\sigma_{1}},\,\boldsymbol{\sigma}_{3}\}$.
Note that these two valleys are not directly connected, and all the
transitions must pass through shallow valleys around $\boldsymbol{m}_{2}$
and $\boldsymbol{m}_{3}$. Hence, to get a precise Markov chain convergence,
we must understand the behavior of the process in these shallow valleys.
If we take $H=H_{2}$, we analyze the transitions between two shallow
valleys $\mathcal{D}_{r_{0}}(\boldsymbol{m}_{2})$ and $\mathcal{D}_{r_{0}}(\boldsymbol{m}_{3})$.
The time scale is now $e^{(H_{2}-h')/\epsilon}$. Finally, if we choose
$H=H_{3}$ , the successive transitions among three valleys $\mathcal{D}_{r_{0}}(\boldsymbol{m}_{4})$,
$\mathcal{D}_{r_{0}}(\boldsymbol{m}_{5})$, and $\mathcal{D}_{r_{0}}(\boldsymbol{m}_{6})$
are investigated in the time scale $e^{(H_{3}-h)/\epsilon}$. Note
that these valleys are not distinguished at the level $H=H_{1}$;
hence, we can analyze the metastable behavior with a higher resolution
by taking this smaller $H$.}
\end{figure}

\subsubsection*{Connection to Eyring--Kramers formula}

Now, we explain the connection between our result and the Eyring--Kramers
formula obtained in \cite{LeeSeo}. For simplicity, we suppose that
$S=S_{\star}$ (i.e., $h_{i}=h$ for all $i\in S$) and that all the
local minima of $U$ are global minima. The explanation below is slightly
more complicated without these assumptions, and we leave the details
to the interested readers. Write $\tau_{\mathcal{A}}$, $\mathcal{A}\subset\mathbb{R}^{d}$,
as the hitting time of the set $\mathcal{A}$. Then, by the Eyring--Kramers
formula obtained in \cite[Theorem 3.5]{LeeSeo}, we have, for $i\in S$
and $\bm{x}\in\mathcal{V}_{i}$,
\[
\begin{aligned}\mathbb{E}_{\bm{x}}^{\epsilon}[\,\tau_{\mathcal{V}_{\star}\setminus\mathcal{V}_{i}}\,]\, & =\,[\,1+o_{\epsilon}(1)\,]\,\frac{\nu_{i}}{\omega_{i}}\,\theta_{\epsilon}\;.\end{aligned}
\]
In other words, for the speeded-up process $\boldsymbol{x}_{\epsilon}(\theta_{\epsilon}\cdot)$
starting from a valley $\mathcal{V}_{i}$, the average of the transition
time to other valleys is approximately $\nu_{i}/\omega_{i}$. This
is in accordance with our result in that the limiting chain $\mathbf{y}(\cdot)$
starting from $i$ jumps to one of the other sites at an average time
of $\big[\,\sum_{j\in S}r_{\mathbf{y}}(i,\,j)\,\big]^{-1}=\nu_{i}/\omega_{i}$.
On the other hand, our result provides more comprehensive information
regarding the metastable behavior compared to the Eyring--Kramers
formula, especially when $S\neq S_{\star}$.

\section{\label{sec_res}Proof Based on Resolvent Approach}

In this section we review the resolvent approach developed in \cite{LMS3}
and then prove Theorems \ref{t33_neg} and \ref{t_main} based on
it.

\subsection{Review of resolvent approach to metastability}

Denote by $L_{\mathbf{y}}$ the generator associated with the limiting
Markov chain $\mathbf{y}(\cdot)$ (defined in the previous section)
that acts on $\mathbf{f}:S_{\star}\rightarrow\mathbb{R}$ such that
\begin{equation}
(L_{\mathbf{y}}\mathbf{f})(i)\,=\,\sum_{j\in S_{\star}}\frac{\beta_{i,\,j}}{\nu_{i}}\,[\,\mathbf{f}(j)-\mathbf{f}(i)\,]\;.\label{e_geny}
\end{equation}
Recall \eqref{conr0} and define, for $i\in S$,
\[
\widehat{\mathcal{V}}_{i}=\bigcup_{\bm{m}\in\mathcal{M}_{i}}\overline{\mathcal{D}_{2r_{\text{0}}}(\bm{m})}\;.
\]
The following analysis of a solution to the resolvent equation is
the main component of the resolvent approach.
\begin{thm}
\label{t_poi}Let $\mathbf{f}:S_{\star}\to\mathbb{R}$ be a given
function and let $\lambda>0$ where both $\mathbf{f}$ and $\lambda$
are independent of $\epsilon$. Then, the unique strong solution $\phi_{\epsilon}^{\mathbf{f}}$
to the resolvent equation (on $u$) on $\mathbb{R}^{d}$
\begin{equation}
(\lambda-\theta_{\epsilon}\mathscr{L}_{\epsilon})\,u\,=\,\sum_{i\in S_{\star}}[(\lambda-L_{\mathbf{y}})\mathbf{f}](i)\,\mathbf{1}_{\mathcal{V}_{i}}\label{res1}
\end{equation}
satisfies
\begin{equation}
\lim_{\epsilon\to\infty}\,\sup_{\bm{x}\in\widehat{\mathcal{V}}_{i}}\,|\,\phi_{\epsilon}^{\mathbf{f}}(\bm{x})-\mathbf{f}(i)\,|\,=\,0\;\;\;\text{for all }i\in S_{\star}\;.\label{e_conphi2}
\end{equation}
\end{thm}

In \cite[Theorem 2.3]{LMS3}, it has been proven that this theorem
implies Theorems \ref{t33_neg} and \ref{t_main} provided that the
underlying metastable process $\boldsymbol{x}_{\epsilon}(\cdot)$
is a Markov process on discrete set. On the other hand, the proof
of Theorem \ref{t_main} based on Theorem \ref{t_poi} requires a
slight technical modification since the solution $\phi_{\epsilon}^{\mathbf{f}}$
obtained in Theorem \ref{t_poi} does not belong to the core of the
generator $\mathscr{L}_{\epsilon}$ associated with the process $\boldsymbol{x}_{\epsilon}(\cdot)$.
We provide the proof here with emphasis on the modification. We remark
that, we took supremum on $\widehat{\mathcal{V}}_{i}$ (instead of
$\mathcal{V}_{i}$ as in \cite{LMS3}) in order to reserve enough
space to carry out this modification.

The main idea is to replace the indicators in the right-hand side
of \eqref{res1} with smooth functions approximating the indicators
so that we can recall the resolvent theory. Then, by using comparison
argument, we shall solve all the technical problems. To that end,
let us take $r_{1}>r_{0}$ such that $r_{1}$ also satisfies the requirement
\eqref{conr0} and define (cf. \eqref{e_Vi})
\[
\mathcal{V}_{i,\,-}=\bigcup_{\bm{m}\in\mathcal{M}_{i}}\overline{\mathcal{D}_{r_{\text{0}}/2}(\bm{m})}\;\;\;\;\text{and \;\;\;\;}\mathcal{V}_{i,\,+}=\bigcup_{\bm{m}\in\mathcal{M}_{i}}\overline{\mathcal{D}_{r_{\text{1}}}(\bm{m})}
\]
so that $\mathcal{V}_{i,\,-}\subset\mathcal{V}_{i}\subset\mathcal{V}_{i,\,+}$.
Then, for each $i\in S_{\star}$, find smooth functions $\zeta_{i,\,-},\,\zeta_{i,\,+}:\mathbb{R}^{d}\rightarrow[0,\,1]$
such that
\begin{equation}
\mathbf{1}_{\mathcal{V}_{i,\,-}}\le\zeta_{i,\,-}\le\mathbf{1}_{\mathcal{V}_{i}}\le\zeta_{i,\,+}\le\mathbf{1}_{\mathcal{V}_{i,\,+}}\;.\label{order}
\end{equation}

The key idea is to consider the functions $\psi_{\epsilon,\,-}^{\mathbf{f}}$
and $\psi_{\epsilon,\,+}^{\mathbf{f}}$ of the equations
\begin{equation}
(\lambda-\theta_{\epsilon}\mathscr{L}_{\epsilon})u\,=\,\sum_{i\in S_{\star}}[(\lambda-L_{\mathbf{y}})\mathbf{f}](i)\,\zeta_{i,\,\pm}\;,\label{res3}
\end{equation}
respectively. Then, $\psi_{\epsilon,\,\pm}^{\mathbf{f}}$ is now a
smooth function that also satisfies \eqref{e_conphi2} in the following
sense.
\begin{prop}
\label{pro_psi}We have that
\begin{equation}
\lim_{\epsilon\to\infty}\,\sup_{\bm{x}\in\mathcal{V}_{i}}\,|\,\psi_{\epsilon,\,\pm}^{\mathbf{f}}(\bm{x})-\mathbf{f}(i)\,|\,=\,0\;\;\;\text{for all }i\in S_{\star}\;.\label{conphi2-1}
\end{equation}
\end{prop}

\begin{proof}
Denote by $\phi_{\epsilon,\,\pm}^{\mathbf{f}}$ the solution to the
equations
\begin{equation}
(\lambda-\theta_{\epsilon}\mathscr{L}_{\epsilon})u\,=\,\sum_{i\in S_{\star}}[(\lambda-L_{\mathbf{y}})\mathbf{f}](i)\,\mathbf{1}_{\mathcal{V}_{i,\,\pm}}\;.\label{res3-2}
\end{equation}
Since Theorem \ref{t_poi} holds for all $r_{0}>0$ satisfying \eqref{conr0},
we can conclude that $\phi_{\epsilon,\,\pm}^{\mathbf{f}}$ also satisfies
\eqref{e_conphi2} in the sense that
\begin{equation}
\lim_{\epsilon\to\infty}\,\sup_{\bm{x}\in\mathcal{V}_{i}}\,|\,\phi_{\epsilon,\,\pm}^{\mathbf{f}}(\bm{x})-\mathbf{f}(i)\,|\,=\,0\;\;\;\text{for all }i\in S_{\star}\;.\label{conphi2}
\end{equation}
Note that the appearance of $\widehat{\mathcal{V}}_{i}$ at \eqref{e_conphi2}
guarantees $\sup_{\bm{x}\in\mathcal{V}_{i}}$ in the previous estimate
for $\phi_{\epsilon,\,-}^{\mathbf{f}}$. Therefore, the statement
of proposition follows from \eqref{order} and the strong positivity
of the operator $\lambda-\theta_{\epsilon}\mathscr{L}_{\epsilon}$.
\end{proof}
Now we use two functions $\psi_{\epsilon,\,-}^{\mathbf{f}}$ and $\psi_{\epsilon,\,+}^{\mathbf{f}}$
to prove Theorems \ref{t33_neg} and \ref{t_main}, respectively.
The huge benefit with these functions is the well-known expressions
\begin{equation}
\psi_{\epsilon,\,\pm}^{\mathbf{f}}(\boldsymbol{x})=\mathbb{E}_{\boldsymbol{x}}^{\epsilon}\left[\int_{0}^{\infty}e^{-\lambda t}G_{\pm}(\boldsymbol{x}_{\epsilon}(t))dt\right]\label{expres}
\end{equation}
where $G_{\pm}=\sum_{i\in S_{\star}}[(\lambda-L_{\mathbf{y}})\mathbf{f}](i)\,\zeta_{i,\,\pm}$
are bounded functions.

\subsection{\label{sec42}Proof of Theorem \ref{t33_neg} }
\begin{proof}[Proof of Theorem \ref{t33_neg}]
For $t>0$, we have
\begin{equation}
\int_{0}^{t}\,\mathbf{1_{\Delta}}(\widetilde{\boldsymbol{x}}_{\epsilon}(s))\,ds\le\int_{0}^{t}\,e^{\lambda t-\lambda s}\,\mathbf{1_{\Delta}}(\widetilde{\boldsymbol{x}}_{\epsilon}(s))\,ds\le e^{\lambda t}\int_{0}^{\infty}\,e^{-\lambda s}\,\mathbf{1_{\Delta}}(\widetilde{\boldsymbol{x}}_{\epsilon}(s))\,ds\ .\label{eq:es1}
\end{equation}
Also, by \eqref{order} and by definition of $\Delta$, we have
\begin{equation}
\mathbf{1_{\Delta}}\le1-\sum_{i\in S_{\star}}\zeta_{i,\,-}\;.\label{eq:es2}
\end{equation}
By \eqref{eq:es1} and \eqref{eq:es2}, the proof of Theorem \ref{t33_neg}
is reduced to show that, for all $i\in S_{\star}$
\[
\lim_{\epsilon\to0}\sup_{\boldsymbol{x}\in\mathcal{V}_{i}}\mathbb{E}_{\boldsymbol{x}}^{\epsilon}\,\Big[\,\int_{0}^{\infty}\,e^{-\lambda s}\Big(1-\sum_{i\in S_{\star}}\zeta_{i,\,-}(\widetilde{\boldsymbol{x}}_{\epsilon}(s))\Big)\,ds\,\Big]\,=\,0
\]
or equivalently
\begin{equation}
\lim_{\epsilon\to0}\sup_{\boldsymbol{x}\in\mathcal{V}_{i}}\mathbb{E}_{\boldsymbol{x}}^{\epsilon}\,\Big[\,\int_{0}^{\infty}\,e^{-\lambda s}\sum_{i\in S_{\star}}\zeta_{i,\,-}(\widetilde{\boldsymbol{x}}_{\epsilon}(s))\,ds\,\Big]\,=\,\frac{1}{\lambda}\;.\label{expres2}
\end{equation}
Note that the constant function $\mathbf{c}:S_{\star}\rightarrow\mathbb{R}$
defined by $\mathbf{c}\equiv\frac{1}{\lambda}$ satisfies $(\lambda-L_{\mathbf{y}})\mathbf{c}\equiv1$
and therefore by \eqref{expres}, we can write
\[
\mathbb{E}_{\boldsymbol{x}}^{\epsilon}\,\Big[\,\int_{0}^{\infty}\,e^{-\lambda s}\sum_{i\in S_{\star}}\zeta_{i,\,-}(\widetilde{\boldsymbol{x}}_{\epsilon}(s))\,ds\,\Big]=\psi_{\epsilon,\,-}^{\mathbf{c}}(\boldsymbol{x})\;.
\]
Therefore, \eqref{expres2} is a direct consequence of Proposition
\ref{pro_psi}.
\end{proof}

\subsection{Proof of Theorem \ref{t_main}}

Next we turn to the proof of Theorem \ref{t_main}. Here we need to
use $\psi_{\epsilon,\,+}^{\mathbf{f}}$ instead. We first recall the
following technical lemma from \cite[Lemma 4.3]{LMS3}.
\begin{lem}
\label{lem3}Theorem \ref{t33_neg} implies that, for all $t>0$,
(cf. \eqref{e_St})
\begin{align*}
 & \lim_{\epsilon\rightarrow0}\sup_{\boldsymbol{x}\in\mathcal{V}_{\star}}\mathbb{E}_{\boldsymbol{x}}^{\epsilon}\left[e^{-\lambda t}-e^{-\lambda S_{\epsilon}(t)}\right]=0\;\;\;\text{and}\\
 & \lim_{\epsilon\rightarrow0}\sup_{\boldsymbol{x}\in\mathcal{V}_{\star}}\mathbb{E}_{\boldsymbol{x}}^{\epsilon}\left[\int_{0}^{t}\left\{ e^{-\lambda s}-e^{-\lambda S_{\epsilon}(s)}\right\} ds\right]=0\;.
\end{align*}
\end{lem}

Now we turn to the proof of Theorem \ref{t_main}.
\begin{proof}[Proof of Theorem \ref{t_main}]
The argument given in \cite{LMS3} uses the solution $\phi_{\epsilon}^{\mathbf{f}}$
to prove Theorem \ref{t_main} when the underlying metastable Markov
process (in our case, $\boldsymbol{x}_{\epsilon}(\cdot)$) is defined
on a discrete set. However, our proof requires a slight modification
since our underlying Markov process $\boldsymbol{x}_{\epsilon}(\cdot)$
is now defined on $\mathbb{R}^{d}$. Technically speaking, the problem
is the fact that $\phi_{\epsilon}^{\mathbf{f}}\notin C^{2}(\mathbb{R}^{d})$
which implies that $\phi_{\epsilon}^{\mathbf{f}}$ does not belong
to the core of the generator $\mathscr{L}_{\epsilon}$. Thus, we cannot
conclude that
\[
M_{\epsilon}^{\phi}(t)=e^{-\lambda t}\phi_{\epsilon}^{\mathbf{f}}(\widetilde{\boldsymbol{x}}_{\epsilon}(t))-\phi_{\epsilon}^{\mathbf{f}}(\widetilde{\boldsymbol{x}}_{\epsilon}(0))+\int_{0}^{t}e^{-\lambda s}(\lambda-\theta_{\epsilon}\mathscr{L}_{\epsilon})\phi_{\epsilon}^{\mathbf{f}}(\widetilde{\boldsymbol{x}}_{\epsilon}(s))ds
\]
is a martingale. This is the only place at which we can not use this
function as in the proof of \cite[Proposition 4.4]{LMS3}. This is
the reason that we introduced the solution $\psi_{\epsilon,\,+}^{\mathbf{f}}$
which belongs to the core of $\mathscr{L}_{\epsilon}$, and we instead
consider
\[
M_{\epsilon}^{\psi}(t)=e^{-\lambda t}\psi_{\epsilon,\,+}^{\mathbf{f}}(\widetilde{\boldsymbol{x}}_{\epsilon}(t))-\psi_{\epsilon,\,+}^{\mathbf{f}}(\widetilde{\boldsymbol{x}}_{\epsilon}(0))+\int_{0}^{t}e^{-\lambda s}(\lambda-\theta_{\epsilon}\mathscr{L}_{\epsilon})\psi_{\epsilon,\,+}^{\mathbf{f}}(\widetilde{\boldsymbol{x}}_{\epsilon}(s))ds
\]
which is now a martingale.

In the proof of \cite[Proposition 4.4]{LMS3} at which $M_{\epsilon}^{\phi}(t)$
was a martingale, the crucial ingredient of the proof is to consider
$M_{\epsilon}^{\phi}(S_{\epsilon}(t))$ which can be rewritten in
a simple form thanks to \eqref{res1} and \eqref{e_conphi2}. Therefore,
if we can show that
\begin{equation}
\lim_{\epsilon\rightarrow0}\sup_{x\in\mathcal{V}_{\star}}\mathbb{E}_{\boldsymbol{x}}^{\epsilon}\left[\left|M_{\epsilon}^{\psi}(S_{\epsilon}(t))-M_{\epsilon}^{\phi}(S_{\epsilon}(t))\right|\right]=0\;\;\;\;\;\text{for all }t\ge0\;,\label{eq:rep}
\end{equation}
we can argue that $M_{\epsilon}^{\phi}(t)$ is a negligible perturbation
of a martingale and therefore the proof of \cite[Proposition 4.4]{LMS3}
can still be applied. To prove \eqref{eq:rep}, we need to prove that
\begin{align}
 & \lim_{\epsilon\rightarrow0}\sup_{x\in\mathcal{V}_{\star}}\mathbb{E}_{\boldsymbol{x}}^{\epsilon}\left|\psi_{\epsilon,\,+}^{\mathbf{f}}(\widetilde{\boldsymbol{x}}_{\epsilon}(S_{\epsilon}(t)))-\phi_{\epsilon}^{\mathbf{f}}(\widetilde{\boldsymbol{x}}_{\epsilon}(S_{\epsilon}(t)))\right|=0\;\;\;\text{for all }t\ge0\text{, and}\label{re1}\\
 & \lim_{\epsilon\rightarrow0}\sup_{x\in\mathcal{V}_{\star}}\mathbb{E}_{\boldsymbol{x}}^{\epsilon}\left[\int_{0}^{S_{\epsilon}(t)}e^{-\lambda s}\mathbf{1}_{\Delta}(\widetilde{\boldsymbol{x}}_{\epsilon}(s))ds\right]=0\label{re2}
\end{align}
where the second one follows from the observation that, for some $C>0$,
\[
\left|(\lambda-\theta_{\epsilon}\mathscr{L}_{\epsilon})\psi_{\epsilon,\,+}^{\mathbf{f}}-(\lambda-\theta_{\epsilon}\mathscr{L}_{\epsilon})\phi_{\epsilon}^{\mathbf{f}}\right|\le C\sum_{i\in S_{\star}}\mathbf{1}_{\mathcal{V}_{i,\,+}\setminus\mathcal{V}_{i}}\le C\mathbf{1}_{\Delta}\;.
\]
Since $\widetilde{\boldsymbol{x}}_{\epsilon}(S_{\epsilon}(t))\in\mathcal{V}_{\star}$
by the definition \eqref{e_St} of $S_{\epsilon}(t)$, the estimate
\eqref{re1} is a direct consequence of \eqref{e_conphi2} and Proposition
\ref{pro_psi}.

Now it remains to prove \eqref{re2}. By the change of variable $s\leftarrow S_{\epsilon}(u)$,
\[
\int_{0}^{S_{\epsilon}(t)}e^{-\lambda s}\mathbf{1}_{\mathcal{V}_{\star}}(\widetilde{\boldsymbol{x}}_{\epsilon}(s))ds=\int_{0}^{t}e^{-\lambda S_{\epsilon}(u)}\mathbf{1}_{\mathcal{V}_{\star}}(\widetilde{\boldsymbol{x}}_{\epsilon}(S_{\epsilon}(u)))du=\int_{0}^{t}e^{-\lambda S_{\epsilon}(u)}du
\]
where the first equality follows from the definition of $S_{\epsilon}(\cdot)$
and the second one follows from $\widetilde{\boldsymbol{x}}_{\epsilon}(S_{\epsilon}(u))\in\mathcal{V}_{\star}$
for all $u\ge0$ by definition. Therefore,
\begin{align*}
\int_{0}^{S_{\epsilon}(t)}e^{-\lambda s}\mathbf{1}_{\Delta}(\widetilde{\boldsymbol{x}}_{\epsilon}(s))ds & =\int_{0}^{S_{\epsilon}(t)}e^{-\lambda s}\left\{ 1-\mathbf{1}_{\mathcal{V}_{\star}}(\widetilde{\boldsymbol{x}}_{\epsilon}(s))\right\} ds\\
 & =\int_{0}^{S_{\epsilon}(t)}e^{-\lambda s}ds-\int_{0}^{t}e^{-\lambda S_{\epsilon}(u)}du\\
 & =\left[\int_{0}^{S_{\epsilon}(t)}e^{-\lambda s}ds-\int_{0}^{t}e^{-\lambda s}ds\right]+\int_{0}^{t}\left\{ e^{-\lambda s}-e^{-\lambda S_{\epsilon}(s)}\right\} ds\;.
\end{align*}
Thus, \eqref{re2} is a direct consequence of Lemma \ref{lem3}.
\end{proof}
\textbf{The remainder of the article is focused on the proof of Theorem
\ref{t_poi}.} Hence, we shall assume in the remainder of the article
that both $\mathbf{f}:S_{\star}\to\mathbb{R}$ and $\lambda>0$ are
fixed and independent of $\epsilon$. Moreover, we simply write $\phi_{\epsilon}$
the solution $\phi_{\epsilon}^{\mathbf{f}}$ of equation \eqref{res1}.
Moreover, we shall always implicitly assume that $\epsilon>0$ is
sufficiently small, as we are focusing on the asymptotics as $\epsilon\rightarrow0$.

\section{\label{sec5}Analysis of Resolvent Equation}

In this section, we prove Theorem \ref{t_poi} up to the construction
of a certain test function, which will be deferred to Sections \ref{sec6}
and \ref{sec7}.

\subsection{\label{sec52}Energy estimate}

In this subsection, we present a crucial energy estimate for the solutions
of\textcolor{black}{{} resolvent equation. Before proceeding to this
estimate, we first remark that $\phi_{\epsilon}$ is a bounded function
as a consequence of \cite[display (4.2)]{LMS3}. A detailed statement
is given as the following proposition. }
\begin{prop}
\label{p_bddsol}There exists $C>0$ so that $\Vert\phi_{\epsilon}\Vert_{L^{\infty}(\mathbb{R}^{d})}<C$
for all $\epsilon>0$.
\end{prop}

\begin{notation}
Here and later, we write $C>0$ as a constant independent of $\epsilon$
and $\boldsymbol{x}$ (of course, $C$ can possibly depend on $\mathbf{f}$
and $\lambda$). Different appearances of $C$ possibly express different
values.
\end{notation}

For sufficiently smooth function $f$, let us define the Dirichlet
form $\mathscr{D}_{\epsilon}(f)$ with respect to the process $\boldsymbol{x}_{\epsilon}(\cdot)$
as
\begin{equation}
\mathscr{D}_{\epsilon}(f)\,:=\,\int_{\mathbb{R}^{d}}f\,(-\mathscr{L}_{\epsilon}f)\,d\mu_{\epsilon}\,=\,\epsilon\,\int_{\mathbb{R}^{d}}\,|\nabla f|^{2}\,d\mu_{\epsilon}\;,\label{e_diri}
\end{equation}
where the latter equality follows from an application of divergence
theorem. Then, the flatness of the solution of resolvent equation
\eqref{res1} on each valley essentially follows from the following
energy estimate (cf. \cite{LMS2,OR,RS}).
\begin{prop}
\label{p_energy}There exists $C>0$ such that, for the solution $\phi_{\epsilon}$
of \eqref{res1},
\begin{equation}
\mathscr{D}_{\epsilon}(\phi_{\epsilon})\,\le\,C\,\theta_{\epsilon}^{-1}\;.\label{e_energy}
\end{equation}
\end{prop}

\begin{proof}
By multiplying both sides of \eqref{res1} by $\phi_{\epsilon}d\mu_{\epsilon}$
and by performing the integral over $\mathbb{R}^{d}$, we get
\[
\int_{\mathbb{R}^{d}}\phi_{\epsilon}(\lambda\phi_{\epsilon}-\theta_{\epsilon}\mathscr{L}_{\epsilon}\phi_{\epsilon})\,d\mu_{\epsilon}\le C
\]
by Proposition \ref{p_bddsol}, since the right-hand side of \eqref{res1}
is a compactly supported bounded function independent of $\epsilon$.
The proof is completed by definition \eqref{e_diri} of the Dirichlet
form.
\end{proof}

\subsection{\label{sec53}Flatness of solution on each well}

We first define
\begin{equation}
\delta\,:=\,\delta(\epsilon)\,=\,\Big(\,\epsilon\,\log\frac{1}{\epsilon}\,\Big)^{1/2}\;,\label{e_delta}
\end{equation}
which is an important scale in the analyses around saddle points carried
out in the next section. Let $J>0$ be a sufficiently large constant,
and let $c_{0}>0$ be a constant that will be specified later in \eqref{e_c0}.
For $i\in S$, define
\begin{equation}
\widehat{\mathcal{W}}_{i}\,:=\,\widehat{\mathcal{W}}_{i,\,\epsilon}\,=\,\{\bm{x}\in\mathcal{W}_{i}:U(\bm{x})\leq H-c_{0}J^{2}\delta^{2}\}\;.\label{e_hatW}
\end{equation}
Note that this set is connected if $\epsilon$ is sufficiently small,
and we have $\mathcal{V}_{i}\subset\widehat{\mathcal{V}}_{i}\subset\widehat{\mathcal{W}}_{i}\subset\mathcal{W}_{i}$.
For $i\in S$, denote by $\mathbf{m}_{\epsilon}(i)$ the average of
$\phi_{\epsilon}$ on $\widehat{\mathcal{W}}_{i}$, i.e.,
\begin{align*}
\mathbf{m}_{\epsilon}(i) & \,=\,\frac{1}{\text{vol}(\widehat{\mathcal{W}}_{i})}\,\int_{\widehat{\mathcal{W}}_{i}}\,\phi_{\epsilon}(\bm{x})\,d\bm{x}\;,
\end{align*}
where $\text{vol}(\mathcal{A})=\int_{\mathcal{A}}d\boldsymbol{x}$
denotes the volume of a Lebesgue measurable set $\mathcal{A\subset\mathbb{R}}^{d}$
with respect to the Lebesgue measure. Remark from Proposition \ref{p_bddsol}
that there exists $C>0$ such that
\begin{equation}
\max_{i\in S}\left|\mathbf{m}_{\epsilon}(i)\right|\le C\label{eq:bdd_m}
\end{equation}
for all $\epsilon>0$. Our next objective is to prove that the function
$\psi_{\epsilon}$ is close to its average value $\mathbf{m}_{\epsilon}(i)$
in $\widehat{\mathcal{W}}_{i}$ in the $L^{\infty}$-sense.
\begin{prop}
\label{p55}For all $i\in S$, we have
\[
\|\phi_{\epsilon}-\mathbf{m}_{\epsilon}(i)\|_{L^{\infty}(\widehat{\mathcal{W}}_{i})}\,=\,o_{\epsilon}(1)\;.
\]
\end{prop}

In \cite[Section 4]{RS}, it has been generally proven that the energy
estimate of the form \eqref{e_energy} is sufficient to prove Proposition
\ref{p55} for the solution $\phi_{\epsilon}$. The argument presented
therein is quite robust, and the reversibility is used only when the
energy estimate is obtained. Hence, the methodology developed in \cite{RS}
can be applied to Proposition \ref{p55} without any modification.
\begin{rem}
In fact, the $L^{\infty}$-boundedness such as Proposition \ref{p_bddsol}
was not available when \cite{RS} started to prove the flatness result
similar to Proposition \ref{p55}. They obtained this boundedness
as a byproduct of the proof of Proposition \ref{p55}. Since we know
this boundedness \textit{a priori }owing to Proposition \ref{p_bddsol},
the proof can indeed be written in even more concise form.
\end{rem}

\subsection{\label{sec54}Characterization of $\mathbf{m}_{\epsilon}$ on deepest
valleys via a test function}

Since $\widehat{\mathcal{V}}_{i}\subset\widehat{\mathcal{W}}_{i}$
for all $i\in S$ by \eqref{conr0}, it remains to prove the following
proposition.
\begin{prop}
\label{p57}We have that
\[
|\,\mathbf{m}_{\epsilon}(i)-\mathbf{f}(i)\,|\,=\,o_{\epsilon}(1)\;\;\;\text{for all }i\in S_{\star}\;.
\]
\end{prop}

Before proving Proposition \ref{p57}, let us formally conclude the
proof of Theorem \ref{t_poi} .
\begin{proof}[Proof of Theorem \ref{t_poi}]
By Propositions \ref{p55} and \ref{p57}, we have $\|\phi_{\epsilon}-\mathbf{f}(i)\|_{L^{\infty}(\widehat{\mathcal{W}}_{i})}=o_{\epsilon}(1)$
for all $i\in S_{\star}$. Since $\widehat{\mathcal{V}}_{i}\subset\widehat{\mathcal{W}}_{i}$,
the proof is completed.
\end{proof}
Now, we turn to Proposition \ref{p57}. The following proposition
is the key in the proof of Proposition \ref{p57}.
\begin{prop}
\label{p_test}Let $\mathbf{g}=\mathbf{g}_{\epsilon}:S\rightarrow\mathbb{R}$
be a function that might depend on $\epsilon$ which is uniformly
bounded in the sense that
\begin{equation}
\sup_{\epsilon>0}\max_{i\in S}|\mathbf{g}(i)|<\infty\;.\label{bddge}
\end{equation}
Then, there exists a uniformly (in $\epsilon$) bounded continuous
function $Q_{\epsilon}^{\mathbf{g}}:\mathbb{R}^{d}\rightarrow\mathbb{R}$
that satisfies, for all $i\in S$,
\begin{align}
 & Q_{\epsilon}^{\mathbf{g}}(\boldsymbol{x})\,\equiv\,\text{\ensuremath{\mathbf{g}(i)}}\text{ \;\;for all \ensuremath{\boldsymbol{x}\in\mathcal{V}_{i}}\;\;\;\;\ensuremath{\text{and}}}\label{e_conQ1}\\
 & \theta_{\epsilon}\,\int_{\mathbb{R}^{d}}Q_{\epsilon}^{\mathbf{g}}\,(\mathscr{L}_{\epsilon}\phi_{\epsilon})\,d\mu_{\epsilon}\,=\,-\frac{1}{\nu_{\star}}\,D_{\mathbf{x}}(\mathbf{g},\mathbf{m}_{\epsilon})+o_{\epsilon}(1)\;.\label{e_conQ2}
\end{align}
\end{prop}

The construction of the test function $Q_{\epsilon}^{\mathbf{g}}$
stated in the proposition above is the most crucial part of the proof
and hence its proof is postponed to the next sections. At this moment,
we prove Proposition \ref{p57} by assuming Proposition \ref{p_test}.
Recall the bi-linear form $D_{\mathbf{x}}(\cdot,\,\cdot)$ defined
in \eqref{e_Dx} and define another bi-linear form $D_{\mathbf{y}}(\mathbf{f},\,\mathbf{g})$
for $\mathbf{f},\,\mathbf{g}:S_{\star}\rightarrow\mathbb{R}$ as
\begin{equation}
D_{\mathbf{y}}(\mathbf{f},\,\mathbf{g})\,:=\,\sum_{i\in S_{\star}}\mathbf{f}(i)\,(-L_{\mathbf{y}}\mathbf{g})(i)\,\frac{\nu_{i}}{\nu_{\star}}\,=\,\frac{1}{\nu_{\star}}\sum_{i\in S}\,\beta_{i,\,j}\,(\mathbf{f}(j)-\mathbf{f}(i))\,(\mathbf{g}(j)-\mathbf{g}(i))\;.\label{e_Dy}
\end{equation}
We recall some relations between $D_{\mathbf{x}}(\cdot,\cdot)$ and
$D_{\mathbf{y}}(\cdot,\cdot)$ proved in \cite{RS}. For $\mathbf{u}:S_{\star}\rightarrow\mathbb{R}$,
we define the harmonic extension $\widetilde{\mathbf{u}}:S\rightarrow\mathbb{R}$
as the extension of $\mathbf{u}$ to $S$ satisfying $(L_{\mathbf{x}}\widetilde{\mathbf{u}})(i)\,=\,0$
for all $i\in S\setminus S_{\star}$.
\begin{lem}
\label{lem59}Let $\mathbf{u},\,\mathbf{v}:S_{\star}\rightarrow\mathbb{R}$
and let $\widetilde{\mathbf{u}}$ and $\widetilde{\mathbf{v}}$ be
the harmonic extensions of $\mathbf{u}$ and $\mathbf{v}$, respectively.
Then, we have $D_{\mathbf{x}}(\widetilde{\mathbf{u}},\,\widetilde{\mathbf{v}})=\nu_{\star}\,D_{\mathbf{y}}(\mathbf{u},\:\mathbf{v})$.
Moreover, for any extensions $\mathbf{v}_{1},\mathbf{v}_{2}$ of $\mathbf{v}$,
we have $D_{\mathbf{x}}(\widetilde{\mathbf{u}},\,\mathbf{v}_{1})=D_{\mathbf{x}}(\widetilde{\mathbf{u}},\,\mathbf{v}_{2})$.
\end{lem}

\begin{proof}
See \cite[Lemma 4.3]{RS}.
\end{proof}
Now, we prove Proposition \ref{p57}.
\begin{proof}[Proof of Proposition \ref{p57}]
Let us define $\mathbf{h}_{\epsilon}:S_{\star}\rightarrow\mathbb{R}$
as
\begin{equation}
\mathbf{h}_{\epsilon}(i)\,:=\,{\bf m}_{\epsilon}(i)-\mathbf{f}(i)\;\;\text{ for all }i\in S_{\star}\;,\label{e_he}
\end{equation}
and let $\widetilde{\mathbf{h}}_{\epsilon}$ be the harmonic extension
of $\mathbf{h}_{\epsilon}$. Then, by the maximum principle and \eqref{eq:bdd_m},
there exists $C>0$ such that
\begin{equation}
\max_{i\in S}\left|\widetilde{\mathbf{h}}_{\epsilon}(i)\right|=\max_{i\in S_{\star}}\left|\mathbf{h}_{\epsilon}(i)\right|\le C\;.\label{bdH}
\end{equation}
Therefore, we can construct a test function $Q_{\epsilon}^{\mathbf{\widetilde{\mathbf{h}}_{\epsilon}}}$
constructed in Proposition \ref{p_test}.

Now, by Proposition \ref{p32}, \eqref{res1} and \eqref{e_conQ1},
we have
\begin{align}
\int_{\mathbb{R}^{d}}Q_{\epsilon}^{\widetilde{\mathbf{h}}_{\epsilon}}\,(\lambda\phi_{\epsilon}-\theta_{\epsilon}\mathscr{L}_{\epsilon}\phi_{\epsilon})\,d\mu_{\epsilon}\, & =\,(1+o_{\epsilon}(1))\sum_{i\in S_{\star}}\mathbf{h}_{\epsilon}(i)\,(\lambda\mathbf{f}-L_{\mathbf{y}}\mathbf{f})(i)\,\frac{\nu_{i}}{\nu_{\star}}\label{eeQl}\\
 & =\lambda\sum_{i\in S_{\star}}\mathbf{h}_{\epsilon}(i)\,\mathbf{f}(i)\,\frac{\nu_{i}}{\nu_{\star}}+D_{\mathbf{y}}(\mathbf{h}_{\epsilon},\,\mathbf{f})+o_{\epsilon}(1)\;,\nonumber
\end{align}
where the last line follows from the definition of $D_{\mathbf{y}}$
and \eqref{bdH}. The crucial idea in the proof is to compute the
left-hand side of \eqref{eeQl} in a different way and to compare
with the previous computation. To that end, we first observe from
Propositions \ref{p32}, \ref{p_bddsol}, \ref{p55} and \eqref{e_conQ1}
that
\begin{equation}
\lambda\int_{\mathbb{R}^{d}}Q_{\epsilon}^{\widetilde{\mathbf{h}}_{\epsilon}}\,\phi_{\epsilon}\,d\mu_{\epsilon}=\lambda\sum_{i\in S_{\star}}\mathbf{h}_{\epsilon}(i)\,\mathbf{m}_{\epsilon}(i)\,\frac{\nu_{i}}{\nu_{\star}}+o_{\epsilon}(1)\;.\label{eeql2}
\end{equation}
By Proposition \ref{p_test} and uniform boundedness of $Q_{\epsilon}^{\widetilde{\mathbf{h}}_{\epsilon}}\,$,
we have
\begin{equation}
\int_{\mathbb{R}^{d}}Q_{\epsilon}^{\widetilde{\mathbf{h}}_{\epsilon}}\,(-\theta_{\epsilon}\mathscr{L}_{\epsilon}\phi_{\epsilon})\,d\mu_{\epsilon}\,=\,\frac{1}{\nu_{\star}}D_{\mathbf{x}}(\widetilde{\mathbf{h}}_{\epsilon},\,{\bf m}_{\epsilon})\,+o_{\epsilon}(1)\;.\label{eeql3}
\end{equation}
Denote by ${\bf m}_{\epsilon}^{\star}:S_{\star}\rightarrow\mathbb{R}$
the restriction of ${\bf m}_{\epsilon}:S\rightarrow\mathbb{R}$ on
$S_{\star}$, and denote by $\widetilde{{\bf m}}_{\epsilon}^{\star}$
the harmonic extension of ${\bf m}_{\epsilon}^{\star}$. Then, by
Lemma \ref{lem59}, we have
\[
D_{\mathbf{x}}(\widetilde{\mathbf{h}}_{\epsilon},\,{\bf m}_{\epsilon})\,=\,D_{\mathbf{x}}(\widetilde{\mathbf{h}}_{\epsilon},\,\widetilde{{\bf m}}_{\epsilon}^{\star})=\nu_{\star}\,D_{\mathbf{y}}({\bf h}_{\epsilon},\,{\bf m}_{\epsilon}^{\star})\;.
\]
Inserting this into \eqref{eeql3} and combining with \eqref{eeql2},
we can conclude that
\[
\int_{\mathbb{R}^{d}}Q_{\epsilon}^{\widetilde{\mathbf{h}}_{\epsilon}}\,(\lambda\phi_{\epsilon}-\theta_{\epsilon}\mathscr{L}_{\epsilon}\phi_{\epsilon})\,d\mu_{\epsilon}=\lambda\sum_{i\in S_{\star}}\mathbf{h}_{\epsilon}(i)\,\mathbf{m}_{\epsilon}(i)\,\frac{\nu_{i}}{\nu_{\star}}+D_{\mathbf{y}}({\bf h}_{\epsilon},\,{\bf m}_{\epsilon}^{\star})+o_{\epsilon}(1)\;.
\]
Comparing this with \eqref{eeQl} and inserting \eqref{e_he}, we
get
\[
\lambda\sum_{i\in S_{\star}}\mathbf{h}_{\epsilon}(i)^{2}\frac{\nu_{i}}{\nu_{\star}}+D_{\mathbf{y}}(\mathbf{h}_{\epsilon},\,\mathbf{h}_{\epsilon})=o_{\epsilon}(1)\;.
\]
This implies that $\max_{i\in S_{\star}}|\mathbf{h}_{\epsilon}(i)|=o_{\epsilon}(1)$
and therefore by recalling the definition \eqref{e_he} of $\mathbf{h}_{\epsilon}$
completes the proof.
\end{proof}

\section{\label{sec6}Construction of Test Function $Q_{\epsilon}^{\mathbf{g}}$}

In this section, we explicitly define the test function $Q_{\epsilon}^{\mathbf{g}}$,
which is an approximating solution to the following elliptic equation:
\begin{equation}
\begin{cases}
\,\mathscr{L}_{\epsilon}^{*}u\,=\,0 & \text{on }\mathbb{R}^{d}\setminus(\cup_{i\in S}\mathcal{V}_{i})\;\;\;\text{and}\\
\,u\,=\,\mathbf{g}(i) & \text{on }\mathcal{V}_{i}\text{ for each }i\in S\;.
\end{cases}\label{e_idea}
\end{equation}
Although we share the same philosophy is this construction with the
reversible case \cite{RS}, the detailed construction and entailed
computations are more complicated compared to the ones therein because
of the non-reversibility.

\subsection{\label{sec61}Neighborhood of saddle points}

To construct the approximating solution to \eqref{e_idea}, we mainly
focus on a neighborhood of each saddle point $\boldsymbol{\sigma}\in\Sigma_{i,\,j}$
for some $i,\,j\in S$, as the function $u$ suddenly changes its
value from $\mathbf{g}(i)$ to $\mathbf{g}(j)$ around such a saddle
point. Therefore, we carefully define several notations regarding
this neighborhood. In this subsection, we fix $i,\,j\in S$ and consider
a saddle point $\bm{\sigma}\in\Sigma_{i,\,j}$. \textbf{\textit{In
addition, we assume that $i<j$ in this subsection.}}
\begin{notation}
\label{not61}We use the following notations in this subsection.
\begin{enumerate}
\item We abbreviate $\mathbb{H}=\mathbb{H}^{\boldsymbol{\sigma}}$ and $\mathbb{L}=\mathbb{L}^{\boldsymbol{\sigma}}$.
\item Since the symmetric matrix $\mathbb{H}$ has only one negative eigenvalue,
we denote by $-\lambda_{1},\,\lambda_{2},\,\dots,\,\lambda_{d}\,(=-\lambda_{1}^{\boldsymbol{\sigma}},\,\lambda_{2}^{\boldsymbol{\sigma}},\,\dots,\,\lambda_{d}^{\boldsymbol{\sigma}})$
the eigenvalues of $\mathbb{H}$, where $-\lambda_{1}$ denotes the
unique negative eigenvalue.
\item Denote by $\boldsymbol{e}_{1}^{\boldsymbol{\sigma}}$ the unit eigenvector
associated with the eigenvalue $-\lambda_{1}$, and by $\boldsymbol{e}_{k}^{\boldsymbol{\sigma}}$,
$k\ge2$, the unit eigenvector associated with the eigenvalue $\lambda_{k}$.
In addition, we assume that the direction of $\boldsymbol{e}_{1}^{\boldsymbol{\sigma}}$
is toward $\mathcal{W}_{i}$, i.e., for all sufficiently small $a>0$,
$\bm{\sigma}+a\boldsymbol{e}_{1}^{\boldsymbol{\sigma}}\in\mathcal{W}_{i}$.
Then, for $\boldsymbol{x}\in\mathbb{R}^{d}$ and $k=1,\,\dots,\,d$,
we write $x_{k}=(\boldsymbol{x}-\boldsymbol{\sigma})\cdot\boldsymbol{e}_{k}^{\boldsymbol{\sigma}}$.
In other words, we have $\boldsymbol{x}=\bm{\sigma}+\sum_{m=1}^{d}x_{m}\bm{e}_{m}^{\boldsymbol{\sigma}}$.
\end{enumerate}
\end{notation}

\begin{figure}

\includegraphics[scale=0.2]{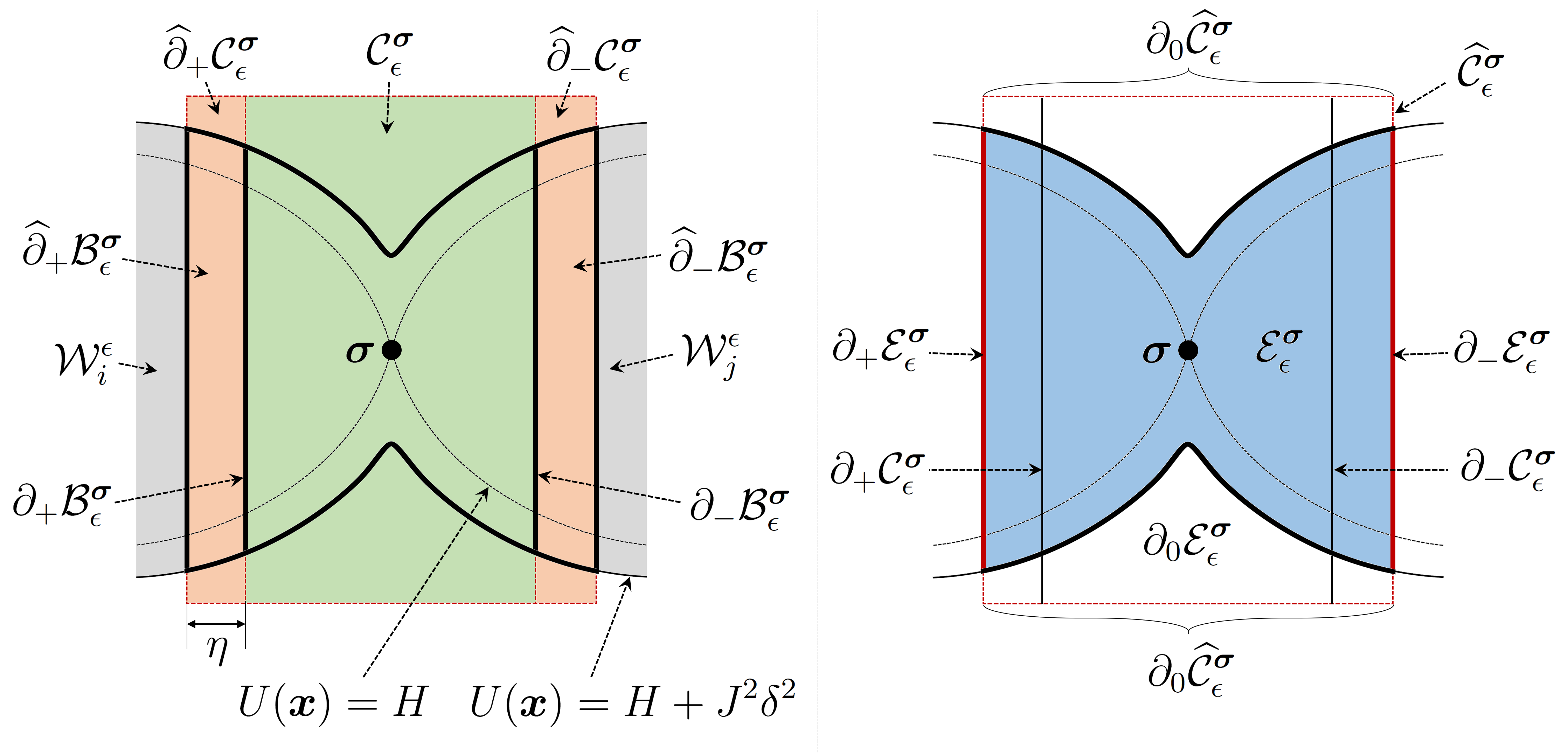}\caption{\label{fig3}Illustration of various sets around a saddle point $\boldsymbol{\sigma}$
introduced in Section \ref{sec61}.}

\end{figure}

Now, we define several sets around $\boldsymbol{\sigma}.$ Figure
\ref{fig3} illustrates the sets appearing in this section. Recall
$\delta$ from \eqref{e_delta} and recall that $J>0$ is a sufficiently
large constant. Define an auxiliary set
\[
\mathcal{T}_{\epsilon}^{\boldsymbol{\sigma}}\,:=\,\Big\{\,\boldsymbol{x}\in\mathbb{R}^{d}:x_{k}\,\in\,\big[\,-\frac{2J\delta}{\lambda_{k}^{1/2}},\,\frac{2J\delta}{\lambda_{k}^{1/2}}\,\big]\,\,\text{ for }\,2\leq k\leq d\,\Big\}\;.
\]
Then, define a box $\mathcal{C}_{\epsilon}^{\bm{\sigma}}$ centered
at $\bm{\sigma}$ as
\[
\mathcal{C}_{\epsilon}^{\bm{\sigma}}\,:=\,\Big\{\,\boldsymbol{x}\in\mathbb{R}^{d}:x_{1}\,\in\,\big[\,-\frac{J\delta}{\lambda_{1}^{1/2}},\,\frac{J\delta}{\lambda_{1}^{1/2}}\,\big]\,\Big\}\,\cap\,\mathcal{T}_{\epsilon}^{\boldsymbol{\sigma}}\;.
\]
The boundary sets $\partial_{+}\mathcal{C}_{\epsilon}^{\bm{\sigma}}$
and $\partial_{-}\mathcal{C}_{\epsilon}^{\bm{\sigma}}$ defined below
will be used later.
\begin{equation}
\partial_{\pm}\mathcal{C}_{\epsilon}^{\bm{\sigma}}\,=\,\Big\{\,\boldsymbol{x}\in\mathcal{C}_{\epsilon}^{\bm{\sigma}}:x_{1}=\pm\frac{J\delta}{\lambda_{1}^{1/2}}\,\Big\}\;.\label{e_bdC}
\end{equation}
We define another scale
\begin{equation}
\eta\,:=\,\eta(\epsilon)\,=\,\epsilon^{2}\;.\label{e_eta}
\end{equation}
Then, define the enlargements of boundaries $\partial_{+}\mathcal{C}_{\epsilon}^{\bm{\sigma}}$
and $\partial_{-}\mathcal{C}_{\epsilon}^{\bm{\sigma}}$ as
\begin{align*}
\widehat{\partial}_{+}\mathcal{C}_{\epsilon}^{\bm{\sigma}}\,=\, & \Big\{\,\boldsymbol{x}\in\mathbb{R}^{d}:x_{1}\,\in\,\big[\,\frac{J\delta}{\lambda_{1}^{1/2}},\,\frac{J\delta}{\lambda_{1}^{1/2}}+\eta\,\big]\,\Big\}\,\cap\,\mathcal{T}_{\epsilon}^{\boldsymbol{\sigma}}\;,\\
\widehat{\partial}_{-}\mathcal{C}_{\epsilon}^{\bm{\sigma}}\,=\, & \Big\{\,\boldsymbol{x}\in\mathbb{R}^{d}:x_{1}\,\in\,\big[\,-\frac{J\delta}{\lambda_{1}^{1/2}}-\eta,\,-\frac{J\delta}{\lambda_{1}^{1/2}}\,\big]\,\Big\}\,\cap\mathcal{\,T}_{\epsilon}^{\boldsymbol{\sigma}}\ .
\end{align*}
With these enlarged boundaries, we can expand $\mathcal{C}_{\epsilon}^{\boldsymbol{\sigma}}$
to
\[
\widehat{\mathcal{C}}_{\epsilon}^{\boldsymbol{\sigma}}\,=\,\mathcal{C}_{\epsilon}^{\bm{\sigma}}\,\cup\,\widehat{\partial}_{+}\mathcal{C}_{\epsilon}^{\bm{\sigma}}\,\cup\,\widehat{\partial}_{-}\mathcal{C}_{\epsilon}^{\bm{\sigma}}\;.
\]
Let
\[
\partial_{0}\mathcal{\widehat{C}}_{\epsilon}^{\bm{\sigma}}\,=\,\Big\{\,\boldsymbol{x}\in\widehat{\mathcal{C}}_{\epsilon}^{\boldsymbol{\sigma}}:x_{k}=\pm\frac{2J\delta}{\lambda_{k}^{1/2}}\,\,\text{ for some }\,2\leq k\leq d\,\Big\}\,\;.
\]
Then, by a Taylor expansion of $U$ around $\boldsymbol{\sigma}$,
we can readily verify that
\begin{equation}
U(\bm{x})\,\geq\,H+\frac{3}{2}\,J^{2}\,\delta^{2}\,[\,1+o_{\epsilon}(1)\,]\text{ \;\;for all }\bm{x}\in\partial_{0}\mathcal{\widehat{C}}_{\epsilon}^{\bm{\sigma}}\;.\label{e_boU}
\end{equation}
For the detailed proof, we refer to \cite[Lemma 8.3]{LeeSeo}. Now,
we define
\[
\mathcal{K}_{\epsilon}\,=\,\{\bm{x}\in\mathbb{R}^{d}:U(\bm{x})<H+J^{2}\delta^{2}\}\;,
\]
so that, by \eqref{e_boU}, the boundary $\partial_{0}\mathcal{\widehat{C}}_{\epsilon}^{\bm{\sigma}}$
does not belong to $\mathcal{K}_{\epsilon}$ provided that $\epsilon$
is sufficiently small. Then, we define
\[
\mathcal{B}_{\epsilon}^{\bm{\sigma}}\,=\,\mathcal{C}_{\epsilon}^{\bm{\sigma}}\cap\mathcal{K}_{\epsilon}\;\;\;\text{,\;\;}\widehat{\partial}_{\pm}\mathcal{B}_{\epsilon}^{\bm{\sigma}}\,=\,\widehat{\partial}_{\pm}\mathcal{C}_{\epsilon}^{\bm{\sigma}}\cap\mathcal{K}_{\epsilon}\;\;\;\text{and\;\;\;}\mathcal{E}_{\epsilon}^{\bm{\sigma}}\,=\,\widehat{\mathcal{C}}_{\epsilon}^{\bm{\sigma}}\cap\mathcal{K}_{\epsilon}\;
\]
so that $\mathcal{E}_{\epsilon}^{\bm{\sigma}}=\mathcal{B}_{\epsilon}^{\bm{\sigma}}\cup\widehat{\partial}_{+}\mathcal{B}_{\epsilon}^{\bm{\sigma}}\cup\widehat{\partial}_{-}\mathcal{B}_{\epsilon}^{\bm{\sigma}}$.
Denote by $\partial\mathcal{E}_{\epsilon}^{\bm{\sigma}}$ the boundary
of the set $\mathcal{E}_{\epsilon}^{\bm{\sigma}}$ and decompose it
into
\[
\partial\mathcal{E}_{\epsilon}^{\bm{\sigma}}\,=\,\partial_{+}\mathcal{E}_{\epsilon}^{\bm{\sigma}}\,\cup\,\partial_{-}\mathcal{E}_{\epsilon}^{\bm{\sigma}}\,\cup\,\partial_{0}\mathcal{E}_{\epsilon}^{\bm{\sigma}}
\]
such that
\begin{align*}
\partial_{\pm}\mathcal{E}_{\epsilon}^{\bm{\sigma}} & \,=\,\Big\{\,\boldsymbol{x}\in\partial\mathcal{E}_{\epsilon}^{\bm{\sigma}}:x_{1}=\pm\,\big(\,\frac{J\delta}{\lambda_{1}^{1/2}}+\eta\,\big)\,\Big\}\,\;\;\text{and\;,}\\
\partial_{0}\mathcal{E}_{\epsilon}^{\bm{\sigma}} & \,=\,\Big\{\,\boldsymbol{x}\in\partial\mathcal{E}_{\epsilon}^{\bm{\sigma}}:x_{1}\neq\pm\,\big(\,\frac{J\delta}{\lambda_{1}^{1/2}}+\eta\,\big)\,\Big\}\;.
\end{align*}
Then, by \eqref{e_boU} (one can readily check from Figure \ref{fig3}),
for sufficiently small $\epsilon>0$,
\begin{equation}
U(\bm{x})\,=\,H+J^{2}\delta^{2}\ \text{\;\;for all}\ \bm{x}\in\partial_{0}\mathcal{E}_{\epsilon}^{\bm{\sigma}}\;.\label{e_bdE1}
\end{equation}
Furthermore, by our selection of the direction of vector $\boldsymbol{e}_{1}^{\boldsymbol{\sigma}}$
(cf. Notation \ref{not61}-(3)), we have
\begin{equation}
\partial_{+}\mathcal{E}_{\epsilon}^{\bm{\sigma}}\,\subset\,\partial\mathcal{W}_{i}^{\epsilon}\;\;\text{and\;\;}\partial_{-}\mathcal{E}_{\epsilon}^{\bm{\sigma}}\,\subset\,\partial\mathcal{W}_{j}^{\epsilon}\;.\label{e_bdE2}
\end{equation}
Similarly, we decompose $\partial\mathcal{B}_{\epsilon}^{\bm{\sigma}}$
into $\partial_{+}\mathcal{B}_{\epsilon}^{\bm{\sigma}}$, $\partial_{-}\mathcal{B}_{\epsilon}^{\bm{\sigma}}$,
and $\partial_{0}\mathcal{B}_{\epsilon}^{\bm{\sigma}}$ such that
\begin{align}
\partial_{\pm}\mathcal{B}_{\epsilon}^{\bm{\sigma}} & \,=\,\Big\{\,\boldsymbol{x}\in\partial\mathcal{B}_{\epsilon}^{\bm{\sigma}}:x_{1}=\pm\frac{J\delta}{\lambda_{1}^{1/2}}\,\Big\}\,\;\;\;\text{and\;\;\;}\partial_{0}\mathcal{B}_{\epsilon}^{\bm{\sigma}}\,=\,\Big\{\,\boldsymbol{x}\in\partial\mathcal{B}_{\epsilon}^{\bm{\sigma}}:x_{1}\neq\pm\frac{J\delta}{\lambda_{1}^{1/2}}\,\Big\}\,\;.\label{e_bdB}
\end{align}

\subsection{\label{sec62}Decomposition of $\mathcal{K}_{\epsilon}$}

Now, we turn to the global picture. Recall $\Sigma^{*}$ from \eqref{e_sig*}.
By \eqref{e_bdE1}, we can observe that $\mathcal{K}_{\epsilon}\setminus(\cup_{\bm{\sigma}\in\Sigma^{*}}\mathcal{E}_{\epsilon}^{\bm{\sigma}})$
consists of $K$ connected components, and we denote by $\mathcal{W}_{i}^{\epsilon}$,
$i\in S$, the component among them containing $\mathcal{V}_{i}$.
Then, we can decompose $\mathcal{K}_{\epsilon}$ such that
\begin{equation}
\mathcal{K}_{\epsilon}\,=\,\Big[\,\bigcup_{i\in S}\mathcal{W}_{i}^{\epsilon}\,\Big]\,\cup\,\Big[\,\bigcup_{\bm{\sigma}\in\Sigma^{*}}\mathcal{E}_{\epsilon}^{\bm{\sigma}}\,\Big]\;.\label{e_decK}
\end{equation}
The test function $Q_{\epsilon}^{\mathbf{g}}$ is constructed on this
global structure of $\mathcal{K}_{\epsilon}$.

\subsection{\label{sec63}Construction of function $Q_{\epsilon}^{\mathbf{g}}$}

\subsubsection*{Construction around a saddle point}

We start by introducing the building block for the construction of
$Q_{\epsilon}^{\mathbf{g}}$, which is a function on $\mathcal{E}_{\epsilon}^{\bm{\sigma}}=\mathcal{B}_{\epsilon}^{\bm{\sigma}}\cup\widehat{\partial}_{+}\mathcal{B}_{\epsilon}^{\bm{\sigma}}\cup\widehat{\partial}_{-}\mathcal{B}_{\epsilon}^{\bm{\sigma}}$.
First, let us focus on the set $\mathcal{B}_{\epsilon}^{\bm{\sigma}}$.
Recall that $\mathbb{H}^{\boldsymbol{\sigma}}+\mathbb{L^{\boldsymbol{\sigma}}}$
has a unique negative eigenvalue $-\mu^{\boldsymbol{\sigma}}$. Denote
by $\mathbb{A}^{\dagger}$ the transpose of the matrix $\mathbb{A}$.
Then, we can readily verify that (cf. \cite[display (8.1)]{LeeSeo})
the matrix $\mathbb{H}^{\boldsymbol{\sigma}}-(\mathbb{L^{\boldsymbol{\sigma}}})^{\dagger}$
is similar to $\mathbb{H}^{\boldsymbol{\sigma}}+\mathbb{L^{\boldsymbol{\sigma}}}$
and hence has the unique negative eigenvalue $-\mu^{\boldsymbol{\sigma}}$.
We denote by $\boldsymbol{v}^{\boldsymbol{\sigma}}$ the unit eigenvector
of $\mathbb{H}^{\boldsymbol{\sigma}}-(\mathbb{L^{\boldsymbol{\sigma}}})^{\dagger}$
associated with $-\mu^{\boldsymbol{\sigma}}$. We assume that $\boldsymbol{v}^{\boldsymbol{\sigma}}\cdot\boldsymbol{e}_{1}^{\boldsymbol{\sigma}}>0$,
as we can take $-\boldsymbol{v}^{\boldsymbol{\sigma}}$ instead if
this inner product is negative. We note that $\boldsymbol{v}^{\boldsymbol{\sigma}}\cdot\boldsymbol{e}_{1}^{\boldsymbol{\sigma}}\neq0$
by \cite[Lemma 8.1]{LeeSeo}.

Define a function $p_{\epsilon}^{\bm{\sigma}}:\mathbb{R}^{d}\rightarrow\mathbb{R}$
as
\begin{equation}
p_{\epsilon}^{\bm{\sigma}}(\bm{x})\,:=\,\frac{1}{C_{\epsilon}^{\boldsymbol{\sigma}}}\int_{-\infty}^{(\bm{x}-\bm{\sigma})\cdot\bm{v}^{\boldsymbol{\sigma}}}\,e^{-\frac{\mu^{\boldsymbol{\sigma}}}{2\epsilon}t^{2}}\,dt\;\;\;\;;\;\boldsymbol{x}\in\mathcal{C}_{\epsilon}^{\boldsymbol{\sigma}}\;,\label{e_pesB}
\end{equation}
where the normalizing constant $C_{\epsilon}^{\boldsymbol{\sigma}}$
is given by
\begin{equation}
C_{\epsilon}^{\boldsymbol{\sigma}}\,=\,\int_{-\infty}^{\infty}\,e^{-\frac{\mu^{\boldsymbol{\sigma}}}{2\epsilon}t^{2}}\,dt\,=\,\sqrt{\frac{2\pi\epsilon}{\mu^{\boldsymbol{\sigma}}}}\;.\label{e_Ces}
\end{equation}
Note that we defined the function on $\mathcal{C}_{\epsilon}^{\boldsymbol{\sigma}}$
containing $\mathcal{B}_{\epsilon}^{\boldsymbol{\sigma}}.$ The function
$p_{\epsilon}^{\bm{\sigma}}$ introduced here is identical to the
one introduced in the companion paper \cite[display (8.8)]{LeeSeo}.
It is remarkable that the test function for the Eyring--Kramers formula
and that for the Markov chain convergence share the building block,
while the global construction from this building block is carried
out in a different manner.

The function $p_{\epsilon}^{\bm{\sigma}}$ is an approximating solution
$\mathscr{L}_{\epsilon}^{*}f\simeq0$ with approximating boundary
conditions $f\simeq1$ on $\partial_{+}\mathcal{C}_{\epsilon}^{\boldsymbol{\sigma}}$
and $f\simeq0$ on $\partial_{-}\mathcal{C}_{\epsilon}^{\boldsymbol{\sigma}}$
by our assumption that $\boldsymbol{v}^{\boldsymbol{\sigma}}\cdot\boldsymbol{e}_{1}^{\boldsymbol{\sigma}}>0$.
The approximating property $\mathscr{L}_{\epsilon}^{*}p_{\epsilon}^{\boldsymbol{\sigma}}\simeq0$
can be quantified in the following proposition, which has been proven
in \cite[Proposition 8.5]{LeeSeo}.
\begin{prop}
\label{p62}For all $\boldsymbol{\sigma}\in\Sigma^{*}$, we have
\[
\theta_{\epsilon}\,\int_{\mathcal{B}_{\epsilon}^{\boldsymbol{\sigma}}}\,|\,\mathscr{L}_{\epsilon}^{*}p_{\epsilon}^{\boldsymbol{\sigma}}\,|\,d\mu_{\epsilon}\,=\,o_{\epsilon}(1)\;.
\]
\end{prop}

Now, we focus on the properties $p_{\epsilon}^{\boldsymbol{\sigma}}\simeq1$
on $\partial_{+}\mathcal{C}_{\epsilon}^{\boldsymbol{\sigma}}$ and
$p_{\epsilon}^{\boldsymbol{\sigma}}\simeq0$ on $\partial_{-}\mathcal{C}_{\epsilon}^{\boldsymbol{\sigma}}$.
When suitably \textcolor{black}{extending this function to get a continuous
function on $\mathbb{R}^{d}$,}\textcolor{red}{{} }these asymptotic
equalities along the boundaries cause technical problems. They become
the exact equality for the reversible case considered in \cite{RS}
as $\boldsymbol{v}^{\boldsymbol{\sigma}}=\boldsymbol{e}_{1}^{\boldsymbol{\sigma}}.$
For our case, the discontinuity is a natural consequence of the non-reversibility;
hence, we need an additional continuation procedure. In \cite{LeeSeo},
this continuation has been carried out by mollification via a smooth
mollifier. For the current problem, such a procedure does not work,
and we take a different path of construction. The enlarged set $\mathcal{E}_{\epsilon}^{\boldsymbol{\sigma}}$
is introduced for performing this continuation procedure.

Now, we continuously extend $p_{\epsilon}^{\boldsymbol{\sigma}}$
to $\widehat{\mathcal{C}}_{\epsilon}^{\bm{\sigma}}$. For each $\bm{x}=\bm{\sigma}+\sum_{k=1}^{d}x_{k}\bm{e}_{k}^{\boldsymbol{\sigma}}\in\widehat{\partial}_{\pm}\mathcal{C}_{\epsilon}^{\bm{\sigma}}$,\textbf{\textcolor{blue}{{}
}}\textcolor{black}{we write
\[
\overline{\boldsymbol{x}}\,=\,\bm{\sigma}\,\pm\,\frac{J\delta}{(\lambda_{1}^{\bm{\sigma}})^{1/2}}\,\bm{e}_{1}^{\bm{\sigma}}\,+\,\sum_{k=2}^{d}x_{k}\,\bm{e}_{k}^{\bm{\sigma}}\,\in\,\partial_{\pm}\mathcal{C}_{\epsilon}^{\bm{\sigma}}\;,
\]
where the boundaries $\partial_{\pm}\mathcal{C}_{\epsilon}^{\bm{\sigma}}$}
are defined in \eqref{e_bdC}. Then, define $p_{\epsilon}^{\bm{\sigma}}$
on the enlarged boundaries $\widehat{\partial}_{\pm}\mathcal{C}_{\epsilon}^{\bm{\sigma}}$
as
\begin{equation}
p_{\epsilon}^{\bm{\sigma}}(\bm{x})\,=\,\begin{cases}
1+\frac{1}{\eta}\,\Big[\,(\bm{x}-\boldsymbol{\sigma})\cdot\bm{e}_{1}^{\bm{\sigma}}-\frac{J\delta}{(\lambda_{1}^{\bm{\sigma}})^{1/2}}-\eta\,\Big]\,(1-p_{\epsilon}^{\bm{\sigma}}(\overline{\boldsymbol{x}})) & \text{for }\bm{x}\in\widehat{\partial}_{+}\mathcal{C}_{\epsilon}^{\bm{\sigma}}\;,\\
\frac{1}{\eta}\,\Big[\,(\bm{x}-\boldsymbol{\sigma})\cdot\bm{e}_{1}^{\bm{\sigma}}+\frac{J\delta}{(\lambda_{1}^{\bm{\sigma}})^{1/2}}+\eta\,\Big]\,p_{\epsilon}^{\bm{\sigma}}(\overline{\boldsymbol{x}}) & \text{for }\bm{x}\in\widehat{\partial}_{-}\mathcal{C}_{\epsilon}^{\bm{\sigma}}\;.
\end{cases}\label{e_pes}
\end{equation}
By such an extension, we can check that $p_{\epsilon}^{\bm{\sigma}}$
is continuous on $\widehat{\mathcal{C}}_{\epsilon}^{\bm{\sigma}}$.
Now, we regard $p_{\epsilon}^{\bm{\sigma}}$ as a function on $\mathcal{E}_{\epsilon}^{\boldsymbol{\sigma}}$.
Then, we can check that $p_{\epsilon}^{\bm{\sigma}}$ satisfies the
exact boundary conditions
\begin{equation}
p_{\epsilon}^{\bm{\sigma}}(\boldsymbol{x})\,=\,\begin{cases}
1 & \text{if }\boldsymbol{x}\in\partial_{+}\mathcal{E}_{\epsilon}^{\boldsymbol{\sigma}}\;,\\
0 & \text{if }\boldsymbol{x}\in\partial_{-}\mathcal{E}_{\epsilon}^{\boldsymbol{\sigma}}\;.
\end{cases}\label{e_pesbd}
\end{equation}
Now, we claim that the cost of this continuation procedure is tolerable.
\begin{lem}
\label{lem63}For all $\boldsymbol{\sigma}\in\Sigma^{*}$, we have
\[
\theta_{\epsilon}\,\epsilon\,\int_{\widehat{\partial}_{\pm}\mathcal{C}_{\epsilon}^{\bm{\sigma}}}\,|\nabla p_{\epsilon}^{\bm{\sigma}}|^{2}\,d\mu_{\epsilon}\,=\,o_{\epsilon}(1)\;.
\]
\end{lem}

We defer the technical proof of this lemma to the next subsection.

\subsubsection*{Global construction}

For $\mathbf{g}=\mathbf{g}_{\epsilon}:S\rightarrow\mathbb{R}$, we
can now define the function $Q_{\epsilon}^{\mathbf{g}}:\mathbb{R}^{d}\to\mathbb{R}$.
First, we define this function on $\mathcal{K}_{\epsilon}$ (cf. \eqref{e_decK})
such that
\begin{equation}
Q_{\epsilon}^{\mathbf{g}}(\boldsymbol{x})\,=\,\begin{cases}
\mathbf{g}(i) & \text{for }\bm{x}\in\mathcal{W}_{i}^{\epsilon},\,i\in S\;,\\
\mathbf{g}(j)+(\mathbf{g}(i)-\mathbf{g}(j))\,p_{\epsilon}^{\bm{\sigma}}(\bm{x}) & \text{for }\bm{x}\in\mathcal{E}_{\epsilon}^{\bm{\sigma}},\,\bm{\sigma}\in\Sigma_{i,\,j}\text{ for}\;i<j\;.
\end{cases}\label{e_Q}
\end{equation}
By \eqref{e_bdE2} and \eqref{e_pesbd}, the function $Q_{\epsilon}^{\mathbf{g}}$
is continuous on $\mathcal{K}_{\epsilon}$. Since for all $\boldsymbol{\sigma}\in\Sigma^{*}$,
it holds that
\[
p_{\epsilon}^{\bm{\sigma}}(\boldsymbol{x})\,\in\,[0,\,1]\;\;\;\text{and\;\;\;|\,}\nabla p_{\epsilon}^{\bm{\sigma}}(\boldsymbol{x})\,|\,\le\,C\,\eta^{-1}\text{\;\;\;for all }\boldsymbol{x}\in\mathcal{E}_{\epsilon}^{\boldsymbol{\sigma}}\;,
\]
we can check that
\[
\lVert\,Q_{\epsilon}^{\mathbf{g}}\,\rVert_{L^{\infty}(\mathcal{K}_{\epsilon})}\,=\,\lVert\,\mathbf{g}\,\rVert_{\infty}\;\;\;\;\text{and\;\;\;\;}\lVert\nabla Q_{\epsilon}^{\mathbf{g}}\rVert_{L^{\infty}(\mathcal{K}_{\epsilon})}\,\leq\,C\,\eta^{-1}\,\lVert\,\mathbf{g}\,\rVert_{\infty}\;,
\]
where $\Vert\mathbf{g}\Vert=\max_{i\in S}|\mathbf{g}(i)|$. Note that
$Q_{\epsilon}^{\mathbf{g}}$ is not differentiable along the boundary
of $\mathcal{E}_{\epsilon}^{\boldsymbol{\sigma}}$ for each $\boldsymbol{\sigma}\in\Sigma^{*}$.
In this computation and subsequent computations, we implicitly regard
$\nabla Q_{\epsilon}^{\mathbf{g}}$ as an a.e. defined function except
for these discontinuity surfaces. Then, we can continuously extend
this function to $\mathbb{R}^{d}$ such that
\begin{equation}
\lVert\,Q_{\epsilon}^{\mathbf{g}}\,\rVert_{L^{\infty}(\mathbb{R}^{d})}\,=\,\lVert\,\mathbf{g}\,\rVert_{\infty}\;\;\;\;\text{and\;\;\;\;}\lVert\,\nabla Q_{\epsilon}^{\mathbf{g}}\,\rVert_{L^{\infty}(\mathbb{R}^{d})}\,\leq\,C\,\eta^{-1}\,\lVert\,\mathbf{g}\,\rVert_{\infty}\;.\label{e_conQ3}
\end{equation}
In particular, we have uniformly boundedness of $Q_{\epsilon}^{\mathbf{g}}$
since we assumed in Proposition \ref{p_test} that $\Vert\mathbf{g}\Vert$
is uniformly bounded in $\epsilon$. Note also that the condition
\eqref{e_conQ1} of Proposition \ref{p_test} is satisfied by $Q_{\epsilon}^{\mathbf{g}}$
immediately from its definition in \eqref{e_Q}. The last and the
most technical part is to check that $Q_{\epsilon}^{\mathbf{g}}$
satisfies \eqref{e_conQ2}. This will be carried out in the next section.
Before doing that, we conclude the proof of Lemma \ref{lem63}.

\subsection{\label{sec64}Proof of Lemma \ref{lem63}}

Before proving Lemma \ref{lem63}, we explain a decomposition of the
extended boundary $\widehat{\partial}_{+}\mathcal{C}_{\epsilon}^{\bm{\sigma}}$,
which will be used several times later. Define, for $a>0$,
\begin{align}
 & \widehat{\partial}_{+}^{\,1,\,a}\mathcal{C}_{\epsilon}^{\boldsymbol{\sigma}}\,=\,\{\,\boldsymbol{x}\in\widehat{\partial}_{+}\mathcal{C}_{\epsilon}^{\boldsymbol{\sigma}}:\overline{\bm{x}}\cdot\bm{v}\ge aJ\delta\,\}\;,\label{ebd_C1}\\
 & \widehat{\partial}_{+}^{\,2,\,a}\mathcal{C}_{\epsilon}^{\boldsymbol{\sigma}}\,=\,\{\,\boldsymbol{x}\in\widehat{\partial}_{+}\mathcal{C}_{\epsilon}^{\boldsymbol{\sigma}}:U(\boldsymbol{x})\geq H+aJ^{2}\delta^{2}\,\}\;.\label{ebd_C2}
\end{align}

\begin{lem}
\label{lem_decC}There exists $a_{0}>0$ such that, for all $a\in(0,\,a_{0})$,
\[
\widehat{\partial}_{+}^{\,1,\,a}\mathcal{C}_{\epsilon}^{\boldsymbol{\sigma}}\cup\widehat{\partial}_{+}^{\,2,\,a}\mathcal{C}_{\epsilon}^{\boldsymbol{\sigma}}\,=\,\widehat{\partial}_{+}\mathcal{C}_{\epsilon}\;.
\]
\end{lem}

The proof is a direct consequence of \cite[Lemma 8.10]{LeeSeo} as
$\eta\ll\delta$ and is omitted.
\begin{proof}[Proof of Lemma \ref{lem63}]
Fix $\boldsymbol{\sigma}\in\Sigma^{*}$, and for convenience of notation,
we assume that $\boldsymbol{\sigma}=\boldsymbol{0}.$ We only consider
the integral on $\widehat{\partial}_{+}\mathcal{C}_{\epsilon}^{\bm{\sigma}}$
since the proof for the case $\widehat{\partial}_{-}\mathcal{C}_{\epsilon}^{\bm{\sigma}}$
is essentially the same. Write $\boldsymbol{e}_{1}^{\boldsymbol{\sigma}}=(e_{1},\,\dots,\,e_{d})$
and $\boldsymbol{v}^{\boldsymbol{\sigma}}=(v_{1},\,\dots,\,v_{d})$.
Then, by the explicit formula \eqref{e_pes} for $p_{\epsilon}^{\bm{\sigma}}$,
we have, for $\boldsymbol{x}\in\widehat{\partial}_{+}\mathcal{C}_{\epsilon}^{\bm{\sigma}}$,
\begin{align}
\nabla_{1}\,p_{\epsilon}^{\bm{\sigma}}(\bm{x}) & \,=\,\frac{e_{1}}{\eta}\,[\,1-p_{\epsilon}^{\bm{\sigma}}(\overline{\boldsymbol{x}})\,]\;,\label{e617}\\
\nabla_{i}\,p_{\epsilon}^{\bm{\sigma}}(\bm{x}) & \,=\,\frac{e_{i}}{\eta}\,[\,1-p_{\epsilon}^{\bm{\sigma}}(\overline{\boldsymbol{x}})\,]\,+\frac{v_{i}^{\boldsymbol{\sigma}}}{\eta\,C_{\epsilon}^{\boldsymbol{\sigma}}}\,e^{-\frac{\mu}{2\epsilon}(\overline{\boldsymbol{x}}\cdot\bm{v}^{\boldsymbol{\sigma}})^{2}}\,\Big[\,\bm{x}\cdot\bm{e}_{1}^{\bm{\sigma}}-\frac{J\delta}{(\lambda_{1}^{\bm{\sigma}})^{1/2}}-\eta\,\Big]\,\;\;\;;\;i\ge2\;.\label{e618}
\end{align}
Since the absolute value of the term in the second pair of brackets
in \eqref{e618} is bounded by $\eta$ for $\boldsymbol{x}\in\widehat{\partial}_{+}\mathcal{C}_{\epsilon}^{\bm{\sigma}}$,
we can conclude that
\[
|\nabla p_{\epsilon}^{\bm{\sigma}}(\bm{x})|^{2}\,\le\,\frac{C}{\eta^{2}}\,[\,1-p_{\epsilon}^{\bm{\sigma}}(\overline{\boldsymbol{x}})\,]^{2}+\frac{C}{\epsilon}\,e^{-\frac{\mu}{\epsilon}(\overline{\boldsymbol{x}}\cdot\bm{v}^{\boldsymbol{\sigma}})^{2}}\;.
\]
Now, for $a\in(0,\,a_{0})$ where $a_{0}$ is the constant appearing
in Lemma \ref{lem_decC}, it suffices to prove that
\begin{equation}
\theta_{\epsilon}\,\epsilon\,\int_{\widehat{\partial}_{+}^{\,k,\,a}\mathcal{C}_{\epsilon}^{\boldsymbol{\sigma}}}\,\Big[\,\frac{1}{\eta^{2}}\,[\,1-p_{\epsilon}^{\bm{\sigma}}(\overline{\boldsymbol{x}})\,]^{2}+\frac{1}{\epsilon}e^{-\frac{\mu}{\epsilon}(\overline{\boldsymbol{x}}\cdot\bm{v}^{\boldsymbol{\sigma}})^{2}}\,\Big]\,\mu_{\epsilon}(\bm{x})\,d\bm{x}\,=\,o_{\epsilon}(1)\;\;\;\text{for }k=1,\,2\;.\label{e619}
\end{equation}

We first consider the case $k=1$. By the elementary inequality $\int_{b}^{\infty}e^{-t^{2}/2}dt\leq\frac{1}{b}e^{-b^{2}/2}$
for $b>0$, we can deduce that
\[
1-p_{\epsilon}^{\bm{\sigma}}(\overline{\boldsymbol{x}})\,\le\,\frac{C}{\overline{\boldsymbol{x}}\cdot\boldsymbol{v}^{\boldsymbol{\sigma}}}\,e^{-\frac{\mu}{2\epsilon}(\overline{\boldsymbol{x}}\cdot\bm{v}^{\boldsymbol{\sigma}})^{2}}\,\le\,\frac{C}{\delta}\,e^{-\frac{\mu}{2\epsilon}(\overline{\boldsymbol{x}}\cdot\bm{v}^{\boldsymbol{\sigma}})^{2}}\;\;\;\text{for}\;\boldsymbol{x}\in\widehat{\partial}_{+}^{\,1,\,a}\mathcal{C}_{\epsilon}^{\boldsymbol{\sigma}}\;.
\]
Hence, the left-hand side of \eqref{e619} is bounded from above by
\begin{equation}
C\frac{\theta_{\epsilon}\epsilon}{\eta^{2}\delta^{2}}\int_{\widehat{\partial}_{+}\mathcal{C}_{\epsilon}^{\boldsymbol{\sigma}}}e^{-\frac{\mu}{\epsilon}(\overline{\boldsymbol{x}}\cdot\bm{v}^{\boldsymbol{\sigma}})^{2}}\mu_{\epsilon}(\bm{x})d\bm{x}\,\le\,C\frac{1}{\epsilon^{d/2+3}\delta^{2}}\int_{\widehat{\partial}_{+}\mathcal{C}_{\epsilon}^{\boldsymbol{\sigma}}}\,e^{-\frac{1}{2\epsilon}\boldsymbol{x}\cdot[\,\mathbb{H}^{\boldsymbol{\sigma}}+2\mu^{\boldsymbol{\sigma}}\,\bm{v}^{\boldsymbol{\sigma}}\otimes\bm{v}^{\boldsymbol{\sigma}}\,]\boldsymbol{x}}d\bm{x}\;,\label{e620}
\end{equation}
where we applied \eqref{e_Zeps}, the Taylor expansion of $U$ around
$\boldsymbol{\sigma}=\boldsymbol{0}$, and the fact that
\[
\begin{split}e^{-\frac{\mu}{\epsilon}(\overline{\boldsymbol{x}}\cdot\bm{v}^{\boldsymbol{\sigma}})^{2}} & \,=\,[\,1+o_{\epsilon}(1)\,]\,e^{-\frac{\mu}{\epsilon}(\boldsymbol{x}\cdot\bm{v}^{\boldsymbol{\sigma}})^{2}}\end{split}
\;.
\]
Since the matrix $\mathbb{H}^{\boldsymbol{\sigma}}+2\mu^{\boldsymbol{\sigma}}\bm{v}^{\boldsymbol{\sigma}}\otimes\bm{v}^{\boldsymbol{\sigma}}$
is positive definite by \cite[Lemma 8.2]{LeeSeo}, we can write
\begin{equation}
\boldsymbol{x}\cdot\,(\mathbb{H}^{\boldsymbol{\sigma}}+2\,\mu^{\boldsymbol{\sigma}}\,\bm{v}^{\boldsymbol{\sigma}}\otimes\bm{v}^{\boldsymbol{\sigma}})\,\boldsymbol{x}\,\ge\,C\,|\boldsymbol{x}|^{2}\,\ge\,C\,J^{2}\,\delta^{2}\ ,\label{e621}
\end{equation}
as there exists $C>0$ such that $|\boldsymbol{x}-\boldsymbol{\sigma}|\ge CJ\delta$
for all $\boldsymbol{x}\in\widehat{\partial}_{+}\mathcal{C}_{\epsilon}^{\boldsymbol{\sigma}}$.
By inserting \eqref{e621} into \eqref{e620}, we can bound the right-hand
side of \eqref{e620} from above by
\[
C\,\frac{1}{\epsilon^{d/2+3}\delta^{2}}\,\text{vol}(\widehat{\partial}_{+}\mathcal{C}_{\epsilon}^{\boldsymbol{\sigma}})\,\epsilon^{CJ^{2}}\,=\,o_{\epsilon}(1)\;,
\]
where the equality holds for sufficiently large $J$ since $\text{vol}(\widehat{\partial}_{+}\mathcal{C}_{\epsilon}^{\boldsymbol{\sigma}})=O(\eta\delta^{d-1})$.

Next, we consider the case $k=2$ of \eqref{e619}. For this case,
by \eqref{e_Zeps}, the left-hand side of \eqref{e619} is bounded
from above by
\[
C\,\epsilon^{1-d/2}\,\int_{\widehat{\partial}_{+}^{\,2,\,a}\mathcal{C}_{\epsilon}^{\boldsymbol{\sigma}\,}}\frac{1}{\eta^{2}}\,e^{-\frac{U(\boldsymbol{x})-H}{\epsilon}}\,d\bm{x}\,\le\frac{C\,\epsilon^{1-d/2}}{\eta^{2}}\,\epsilon^{aJ^{2}}\,\text{vol}(\widehat{\partial}_{+}\mathcal{C}_{\epsilon}^{\boldsymbol{\sigma}})\,=\,o_{\epsilon}(1)\;,
\]
where the inequality holds from the definition of $\widehat{\partial}_{+}^{\,2,\,a}\mathcal{C}_{\epsilon}^{\boldsymbol{\sigma}}$,
and the last equality holds for sufficiently large $J$ since $\text{vol}(\widehat{\partial}_{+}\mathcal{C}_{\epsilon}^{\boldsymbol{\sigma}})=O(\eta\delta^{d-1})$.
This completes the proof.
\end{proof}

\section{\label{sec7}Proof of Proposition \ref{p_test}}

We fix $\mathbf{g}=\mathbf{g}_{\epsilon}:S\rightarrow\mathbb{R}$
throughout this section which is\emph{ uniformly bounded in $\epsilon$}
in the sense of \eqref{bddge}. The function $Q_{\epsilon}^{\mathbf{g}}$
appearing in this section is the one defined in \eqref{e_Q}.

\subsection{\label{sec71}Reduction to local computations around saddle points}

In this subsection, we reduce the proof of Proposition \ref{p_test}
to two local estimates around saddle point $\boldsymbol{\sigma}\in\Sigma^{*}$.
We perform this reduction via the following proposition.
\begin{prop}
\label{p71}It holds that\textbf{ }
\begin{align*}
 & \theta_{\epsilon}\,\int_{\mathbb{R}^{d}}Q_{\epsilon}^{\mathbf{g}}\,(\mathscr{L}_{\epsilon}\phi_{\epsilon})\,d\mu_{\epsilon}\\
=\, & o_{\epsilon}(1)\,+\sum_{i,\,j\in S,\,i<j}\,\sum_{\boldsymbol{\sigma}\in\Sigma_{i,\,j}}(\,\mathbf{g}(i)-\mathbf{g}(j)\,)\,[\,A_{1}(\boldsymbol{\sigma})+A_{2}^{+}(\boldsymbol{\sigma})+A_{2}^{-}(\boldsymbol{\sigma})\,]\;,
\end{align*}
where
\begin{align*}
A_{1}(\boldsymbol{\sigma}) & \,=\,-\theta_{\epsilon}\,\epsilon\,\int_{\mathcal{B}_{\epsilon}^{\bm{\sigma}}}\,\nabla p_{\epsilon}^{\boldsymbol{\sigma}}\cdot\,\Big[\,\nabla\phi_{\epsilon}-\frac{1}{\epsilon}\phi_{\epsilon}\boldsymbol{\ell}\,\Big]\,d\mu_{\epsilon}\;\;\text{\;and}\;\\
A_{2}^{\pm}(\boldsymbol{\sigma}) & \,=\,\theta_{\epsilon}\,\int_{\widehat{\partial}_{\pm}\mathcal{B}_{\epsilon}^{\bm{\sigma}}}\,\phi_{\epsilon}\,(\nabla p_{\epsilon}^{\boldsymbol{\sigma}}\cdot\boldsymbol{\ell})\,d\mu_{\epsilon}\;.
\end{align*}
\end{prop}

\begin{proof}
By the divergence theorem, we can write
\begin{align*}
\theta_{\epsilon}\int_{\mathbb{R}^{d}}Q_{\epsilon}^{\mathbf{g}}\,(\mathscr{L}_{\epsilon}\phi_{\epsilon})\,d\mu_{\epsilon}\,= & \,-\theta_{\epsilon}\,\epsilon\,\int_{\mathbb{R}^{d}}\,\nabla Q_{\epsilon}^{\mathbf{g}}\cdot\,\Big[\,\nabla\phi_{\epsilon}-\frac{1}{\epsilon}\phi_{\epsilon}\boldsymbol{\ell}\,\Big]\,d\mu_{\epsilon}\;.
\end{align*}
Note that we can apply the divergence theorem since $Q_{\epsilon}^{\mathbf{g}}$
is continuous, while its gradient is not defined along $\partial\mathcal{E}_{\epsilon}^{\boldsymbol{\sigma}}$.
Now, let us investigate the right-hand side. First, note that $\nabla Q_{\epsilon}^{\mathbf{g}}\equiv0$
on $\mathcal{W}_{i}^{\epsilon}$ by definition. Hence, we can rewrite
the right-hand side as
\begin{equation}
-\theta_{\epsilon}\,\epsilon\,\Big(\,\int_{\mathbb{R}^{d}\setminus\mathcal{K}_{\epsilon}}+\sum_{\bm{\sigma}\in\Sigma^{*}}\int_{\mathcal{\mathcal{E}}_{\epsilon}^{\bm{\sigma}}}\,\Big)\,\nabla Q_{\epsilon}^{\mathbf{g}}\cdot\,\Big[\,\nabla\phi_{\epsilon}-\frac{1}{\epsilon}\phi_{\epsilon}\boldsymbol{\ell}\,\Big]\,d\mu_{\epsilon}\;.\label{e_decs}
\end{equation}
Now we consider two integrals separately.

\smallskip{}

\noindent\textbf{First integral of \eqref{e_decs}:} We will separately
show that
\begin{align}
\theta_{\epsilon}\,\epsilon\,\int_{\mathbb{R}^{d}\setminus\mathcal{K}_{\epsilon}}\,(\nabla Q_{\epsilon}^{\mathbf{g}}\cdot\nabla\phi_{\epsilon})\,d\mu_{\epsilon} & \,=\,o_{\epsilon}(1)\;\;\;\text{and}\;\label{e_negQ1}\\
\theta_{\epsilon}\,\int_{\mathbb{R}^{d}\setminus\mathcal{K}_{\epsilon}}\,(\nabla Q_{\epsilon}^{\mathbf{g}}\cdot\boldsymbol{\ell})\,\phi_{\epsilon}\,d\mu_{\epsilon} & \,=\,o_{\epsilon}(1)\;.\label{e_negQ2}
\end{align}
For \eqref{e_negQ1}, by \eqref{e_conQ3} and the Cauchy--Schwarz
inequality,
\begin{equation}
\Big|\,\theta_{\epsilon}\,\epsilon\,\int_{\mathbb{R}^{d}\setminus\mathcal{K}_{\epsilon}}\,\nabla Q_{\epsilon}^{\mathbf{g}}\cdot\nabla\phi_{\epsilon}\,d\mu_{\epsilon}\,\Big|^{2}\,\le\,C\,\theta_{\epsilon}^{2}\,\epsilon\,\frac{1}{\eta^{2}}\,\mu_{\epsilon}(\mathbb{R}^{d}\setminus\mathcal{K}_{\epsilon})\,\mathscr{D}_{\epsilon}(\phi_{\epsilon})\;.\label{e_negQ3}
\end{equation}
Note that the uniform boundedness of $\mathbf{g}$ is implicitly used
here. We shall repeatedly use this boundedness in the later arguments
as well. Since $U>H+\delta^{2}J^{2}$ on $\mathbb{R}^{d}\setminus\mathcal{K}_{\epsilon}$,
by \eqref{e_tight} and Proposition \ref{p32},
\begin{equation}
\mu_{\epsilon}(\mathbb{R}^{d}\setminus\mathcal{K}_{\epsilon})\,\le\,\frac{C}{Z_{\epsilon}}\,e^{-\frac{H+\delta^{2}J^{2}}{\epsilon}}\,\le\,C\,\theta_{\epsilon}^{-1}\,\epsilon^{J^{2}-\frac{d}{2}}\;.\label{e_negQ4}
\end{equation}
Applying this and Proposition \ref{p_energy} to \eqref{e_negQ3}
yields
\[
\Big|\,\theta_{\epsilon}\,\epsilon\,\int_{\mathbb{R}^{d}\setminus\mathcal{K}_{\epsilon}}\,(\nabla Q_{\epsilon}^{\mathbf{g}}\cdot\nabla\phi_{\epsilon})\,d\mu_{\epsilon}\,\Big|\,\le\,\frac{C}{\eta}\,\epsilon^{\frac{1}{2}J^{2}-\frac{d}{4}+\frac{1}{2}}\;.
\]
Since $\eta=\epsilon^{2}$, we obtain \eqref{e_negQ1} by taking $J$
to be sufficiently large.

Now, we turn to \eqref{e_negQ2}. By \eqref{e_conQ3}, Proposition
\ref{p_bddsol}, \eqref{e_negQ4} we can bound the absolute value
of the left-hand side of \eqref{e_negQ2} by $C\theta_{\epsilon}\eta^{-1}\mu_{\epsilon}(\mathbb{R}^{d}\setminus\mathcal{K}_{\epsilon})=o_{\epsilon}(1)$.

\noindent\textbf{Second integral of \eqref{e_decs}:} For each $\boldsymbol{\sigma}\in\Sigma_{i,\,j}$
with $i,\,j\in S$, we have $\nabla Q_{\epsilon}^{\mathbf{g}}=(\mathbf{g}(i)-\mathbf{g}(j))\nabla p_{\epsilon}^{\boldsymbol{\sigma}}$
on $\mathcal{E}_{\epsilon}^{\bm{\sigma}}$. Therefore, we can prove
that this integral is $A_{1}(\boldsymbol{\sigma})+A_{2}^{+}(\boldsymbol{\sigma})+A_{2}^{-}(\boldsymbol{\sigma})$
provided that we can prove
\[
\theta_{\epsilon}\,\epsilon\,\int_{\widehat{\partial}_{\pm}\mathcal{B}_{\epsilon}^{\bm{\sigma}}}\,(\nabla p_{\epsilon}^{\boldsymbol{\sigma}}\cdot\nabla\phi_{\epsilon})\,d\mu_{\epsilon}\,=\,o_{\epsilon}(1)\;.
\]
This follows from the Cauchy--Schwarz inequality, Proposition \ref{p_energy},
and Lemma \ref{lem63}.
\end{proof}
Based on the previous proposition, it suffices to estimate $A_{1}(\boldsymbol{\sigma})$
and $A_{2}^{\pm}(\boldsymbol{\sigma})$ for each $\boldsymbol{\sigma}\in\Sigma^{*}$.
These estimates are carried out via the following proposition.
\begin{prop}
\label{p72}For $i,\,j\in S$ with $i<j$ and $\boldsymbol{\sigma}\in\Sigma_{i,\,j}$,
we have
\begin{equation}
A_{1}(\boldsymbol{\sigma})\,=\,\frac{\lambda_{1}^{\boldsymbol{\sigma}}}{2\pi\nu_{\star}\sqrt{-\det\mathbb{H}^{\boldsymbol{\sigma}}}}\,(\mathbf{m}_{\epsilon}(j)-\mathbf{m}_{\epsilon}(i))\,+o_{\epsilon}(1)\;,\label{e_A1}
\end{equation}
and
\begin{align}
A_{2}^{+}(\boldsymbol{\sigma}) & \,=\,-\frac{\lambda_{1}^{\boldsymbol{\sigma}}}{2\pi\nu_{\star}\sqrt{-\det\mathbb{H}^{\boldsymbol{\sigma}}}}\,\frac{(\mathbb{L}^{\boldsymbol{\sigma}}\,(\mathbb{H^{\boldsymbol{\sigma}}})^{-1}\,\bm{v}^{\boldsymbol{\sigma}})\cdot\bm{e}_{1}^{\boldsymbol{\sigma}}}{\bm{v}^{\boldsymbol{\sigma}}\cdot\bm{e}_{1}^{\boldsymbol{\sigma}}}\,\mathbf{m}_{\epsilon}(i)+o_{\epsilon}(1)\;,\label{e_A2+}\\
A_{2}^{-}(\boldsymbol{\sigma}) & \,=\,+\frac{\lambda_{1}^{\boldsymbol{\sigma}}}{2\pi\nu_{\star}\sqrt{-\det\mathbb{H}^{\boldsymbol{\sigma}}}}\,\frac{(\mathbb{L}^{\boldsymbol{\sigma}}\,(\mathbb{H^{\boldsymbol{\sigma}}})^{-1}\,\bm{v}^{\boldsymbol{\sigma}})\cdot\bm{e}_{1}^{\boldsymbol{\sigma}}}{\bm{v}^{\boldsymbol{\sigma}}\cdot\bm{e}_{1}^{\boldsymbol{\sigma}}}\,\mathbf{m}_{\epsilon}(j)+o_{\epsilon}(1)\;.\label{e_A2-}
\end{align}
\end{prop}

Before proving this proposition, we conclude the demonstration of
Proposition \ref{p_test} by assuming this proposition.
\begin{proof}[Proof of Proposition \ref{p_test}]
First, we check that
\begin{align*}
 & (\bm{v}^{\boldsymbol{\sigma}}+\mathbb{L}^{\boldsymbol{\sigma}}\,(\mathbb{H^{\boldsymbol{\sigma}}})^{-1}\,\bm{v}^{\boldsymbol{\sigma}})\cdot\bm{e}_{1}^{\boldsymbol{\sigma}}\,=\,(\mathbb{I}-(\mathbb{H^{\boldsymbol{\sigma}}})^{-1}\,(\mathbb{L}^{\boldsymbol{\sigma}})^{\dagger})\,\bm{v}^{\boldsymbol{\sigma}}\cdot\bm{e}_{1}^{\boldsymbol{\sigma}}\\
 & \;=\,(\mathbb{H^{\boldsymbol{\sigma}}})^{-1}\,(\mathbb{H}^{\boldsymbol{\sigma}}-(\mathbb{L}^{\boldsymbol{\sigma}})^{\dagger})\,\bm{v}^{\boldsymbol{\sigma}}\cdot\bm{e}_{1}^{\boldsymbol{\sigma}}\,=\,-\mu^{\boldsymbol{\sigma}}\,(\mathbb{H^{\boldsymbol{\sigma}}})^{-1}\,\bm{v}^{\boldsymbol{\sigma}}\cdot\bm{e}_{1}^{\boldsymbol{\sigma}}\,=\,\frac{\mu^{\boldsymbol{\sigma}}}{\lambda_{1}^{\boldsymbol{\sigma}}}\,(\bm{v}^{\boldsymbol{\sigma}}\cdot\bm{e}_{1}^{\boldsymbol{\sigma}})\;,
\end{align*}
where the first identity follows from the fact that the matrix $\mathbb{H}^{\boldsymbol{\sigma}}\mathbb{L}^{\boldsymbol{\sigma}}$
is a skew-symmetric matrix by \cite[Lemma 4.5]{LeeSeo}, and the last
identity follows from the fact that $\boldsymbol{e}_{1}^{\boldsymbol{\sigma}}$
is the eigenvector of $\mathbb{H}^{\boldsymbol{\sigma}}$ associated
with the eigenvalue $-\lambda_{1}^{\boldsymbol{\sigma}}.$ We can
combine this computation with Proposition \ref{p72} to get
\begin{equation}
A_{1}(\boldsymbol{\sigma})+A_{2}^{+}(\boldsymbol{\sigma})+A_{2}^{-}(\boldsymbol{\sigma})\,=\,\frac{\omega^{\boldsymbol{\sigma}}}{\nu_{\star}}\,(\mathbf{m}_{\epsilon}(j)-\mathbf{m}_{\epsilon}(i))+o_{\epsilon}(1)\;,\label{e79}
\end{equation}
where the Eyring--Kramers constant $\omega^{\boldsymbol{\sigma}}$
is defined in \eqref{e_EKconst}, and we implicitly used \eqref{eq:bdd_m}.
Inserting \eqref{e79} into Proposition \ref{p71} completes the proof.
\end{proof}
Now, it remains to prove Proposition \ref{p72}. We provide the estimates
of $A_{1}(\boldsymbol{\sigma})$ and $A_{2}^{\pm}(\boldsymbol{\sigma})$
in Sections \ref{sec73} and \ref{sec74}, respectively.

\subsection{\label{sec72}A Change of coordinate on $\partial_{+}\mathcal{B}_{\epsilon}^{\boldsymbol{\sigma}}$}

Hereafter, it suffices to focus only on a single saddle point $\boldsymbol{\sigma}$;
hence, in the remainder of the current section, we recall Notation
\ref{not61} and use the following conventions: \textit{we fix $\boldsymbol{\sigma}\in\Sigma_{i,\,j}$
for some $i,\,j\in S$ with $i<j$, assume that $\boldsymbol{\sigma}=\boldsymbol{0}$
for simplicity of notation, and drop the superscript $\boldsymbol{\sigma}$
from the notations, e.g., we write $p_{\epsilon}$ and $\mathcal{B}_{\epsilon}$
instead of $p_{\epsilon}^{\boldsymbol{\sigma}}$ and $\mathcal{B}_{\epsilon}^{\boldsymbol{\sigma}}$,
respectively. }

Before proceeding to the proof of Proposition \ref{p72}, we recall
in this subsection a change of coordinate introduced in \cite[Section 8.4]{LeeSeo},
which maps $\partial_{+}\mathcal{B}_{\epsilon}$ to a subset of $\mathbb{R}^{d-1}$.
For $\mathbb{A}\in\mathbb{R}^{d\times d}$ and $\boldsymbol{u}=(u_{1},\,\dots,\,u_{d})\in\mathbb{R}^{d}$,
we define $\widetilde{\mathbb{A}}\in\mathbb{R}^{(d-1)\times(d-1)}$
and $\widetilde{\boldsymbol{u}}\in\mathbb{R}^{d-1}$ as
\begin{equation}
\widetilde{\mathbb{A}}\,:=\,(\mathbb{A}_{i,\,j})_{2\leq i,\,j\leq d}\;\;\;\text{and\;\;\;}\widetilde{\boldsymbol{u}}\,:=\,(u_{2},\,\dots,\,u_{d})\;,\label{e_tildes}
\end{equation}
respectively. Define a vector $\boldsymbol{\gamma}=(\gamma_{2},\,\dots,\,\gamma_{d})\in\mathbb{R}^{d-1}$
by\textbf{ $\gamma_{k}=\frac{v_{k}}{v_{1}}\cdot\frac{\lambda_{1}^{1/2}}{\lambda_{k}}J\delta$
}for $2\le k\le d$, where $\boldsymbol{v}=\boldsymbol{v}^{\boldsymbol{\sigma}}=(v_{1},\,\dots,\,v_{d})$
denotes the eigenvector introduced in Section \ref{sec63}, at which
it has been mentioned that $v_{1}=\boldsymbol{v}\cdot\boldsymbol{e}\neq0$.
Define
\[
\mathcal{P}_{\delta}\,:=\,\Big\{\,\bm{x}=(x_{1},\,\dots,\,x_{d})\in\mathbb{R}^{d}:x_{1}=\frac{J\delta}{\lambda_{1}^{1/2}}\,\Big\}\,\subset\,\mathbb{R}^{d}\;,
\]
so that $\partial_{+}\mathcal{B}_{\epsilon},\,\partial_{+}\mathcal{C}_{\epsilon}\subset\mathcal{P}_{\delta}$,
and define a map $\Pi_{\epsilon}:\mathcal{P}_{\delta}\rightarrow\mathbb{R}^{d-1}$
by $\Pi_{\epsilon}(\boldsymbol{x})=\widetilde{\boldsymbol{x}}+\boldsymbol{\gamma}$.
This maps the change of coordinate from $\partial_{+}\mathcal{B}_{\epsilon}$
to $\mathbb{R}^{d-1}$, which simplifies computations significantly.
For instance, we have the following result.
\begin{lem}
\label{lem73}For all $\boldsymbol{x}\in\partial_{+}\mathcal{B}_{\epsilon}$,
we have
\[
\bm{x}\cdot(\mathbb{H}+\mu\,\bm{v}\otimes\bm{v})\,\bm{x}\,=\,\Pi_{\epsilon}(\boldsymbol{x})\cdot(\widetilde{\mathbb{H}}+\mu\,\widetilde{\boldsymbol{v}}\otimes\widetilde{\boldsymbol{v}})\,\Pi_{\epsilon}(\boldsymbol{x})\;.
\]
In addition, the matrix $\widetilde{\mathbb{H}}+\mu\,\widetilde{\bm{v}}\otimes\widetilde{\bm{v}}$
is positive definite and \textbf{
\[
\det\,(\widetilde{\mathbb{H}}+\mu\,\widetilde{\bm{v}}\otimes\widetilde{\bm{v}})\,=\,(\boldsymbol{v}\cdot\boldsymbol{e}_{1})^{2}\,\frac{\mu}{\lambda_{1}}\,\prod_{k=2}^{d}\lambda_{k}\;.
\]
}
\end{lem}

\begin{proof}
We refer to \cite[Lemmas 8.7 and 8.9]{LeeSeo} for the proof.
\end{proof}
In \cite[Lemma 8.8]{LeeSeo}, it has been verified that the image
of $\Pi_{\epsilon}(\partial_{+}\mathcal{B}_{\epsilon}^{\boldsymbol{\sigma}})$
is comparable with a ball in $\mathbb{R}^{d-1}$ centered at the origin
with radius of order $\delta$. In the next lemma, we slightly strengthen
this result. Recall the definition of $\widehat{\mathcal{W}}_{i}$
from \eqref{e_hatW}.
\begin{lem}
\label{lem74}There exist constants $r,\,R>0$ such that
\[
\mathcal{D}_{r\delta}^{(d-1)}\,\subset\,\Pi_{\epsilon}(\partial_{+}\mathcal{B}_{\epsilon}^{\boldsymbol{\sigma}}\cap\widehat{\mathcal{W}}_{i})\,\subset\,\Pi_{\epsilon}(\partial_{+}\mathcal{B}_{\epsilon}^{\boldsymbol{\sigma}})\,\subset\,\mathcal{D}_{R\delta}^{(d-1)}\;,
\]
where $\mathcal{D}_{a}^{(d-1)}$ represents a ball in $\mathbb{R}^{d-1}$
of radius $a$ centered at the origin.
\end{lem}

\begin{proof}
In view of Lemma \cite[Lemma 8.8]{LeeSeo}, it suffices to show the
first inclusion. Define
\begin{align*}
\mathcal{P}_{\delta} & \,:=\,\Big\{\,\bm{x}=(x_{1},\,\dots,\,x_{d})\in\mathbb{R}^{d}:x_{1}=\frac{J\delta}{\lambda_{1}^{1/2}}\,\Big\}\,\subset\mathbb{R}^{d}\;\;\;\text{and}\\
\overline{\boldsymbol{\gamma}} & \,:=\,\Big(\,\frac{J\delta}{\lambda_{1}^{1/2}},\,-\gamma_{2},\,\cdots,\,-\gamma_{d}\,\Big)\,\in\mathcal{P}_{\delta}\;.
\end{align*}
Then, it has been shown in \cite[display (8.16)]{LeeSeo} that
\begin{equation}
U(\overline{\boldsymbol{\gamma}})\,=\,H-\frac{\lambda_{1}}{2\,\mu\,v_{1}^{2}}\,J^{2}\,\delta^{2}+O(\delta^{3})\,<\,H-2\,c_{0}\,J^{2}\,\delta^{2}\label{e_c0}
\end{equation}
for all sufficiently small $\epsilon>0$ if we take $c_{0}>0$ as
a sufficiently small constant. By inserting this $c_{0}$ into the
definition of $\widehat{\mathcal{W}}_{i}$ in \eqref{e_hatW}, we
can find sufficiently small $r>0$ such that $\mathcal{D}_{r\delta}(\overline{\boldsymbol{\gamma}})\cap\mathcal{P}_{\delta}\subset\partial_{+}\mathcal{B}_{\epsilon}^{\boldsymbol{\sigma}}\cap\widehat{\mathcal{W}}_{i}$.
Since $\Pi_{\epsilon}(\overline{\boldsymbol{\gamma}})=\bm{0}$, we
have $\mathcal{D}_{r\delta}^{(d-1)}=\Pi_{\epsilon}(\mathcal{D}_{r\delta}(\overline{\boldsymbol{\gamma}})\cap\mathcal{P}_{\delta})$
and the proof is completed.
\end{proof}

\subsection{\label{sec73}Estimate of $A_{1}(\boldsymbol{\sigma})$}

Now, we prove \eqref{e_A1} of Proposition \ref{p72}. By the divergence
theorem, we can write
\[
A_{1}(\boldsymbol{\sigma})\,=\,-\theta_{\epsilon}\,\epsilon\,\int_{\partial\mathcal{B}_{\epsilon}}\,\phi_{\epsilon}\,[\,\nabla p_{\epsilon}\cdot\bm{n}_{\mathcal{B}_{\epsilon}}\,]\,\sigma(d\mu_{\epsilon})+\theta_{\epsilon}\,\int_{\mathcal{B}_{\epsilon}}(\mathscr{L}_{\epsilon}^{*}p_{\epsilon})\,\phi_{\epsilon}\,d\mu_{\epsilon}\;.
\]
By Propositions \ref{p_bddsol} and \ref{p62}, we can observe that
the second term at the right-hand side is $o_{\epsilon}(1)$. Therefore,
we can write
\begin{equation}
A_{1}(\boldsymbol{\sigma})\,=\,K_{0}+K_{+}+K_{-}+o_{\epsilon}(1)\;,\label{e_A1dec}
\end{equation}
where
\begin{align*}
K_{0} & \,=\,-\theta_{\epsilon}\,\epsilon\,\int_{\partial_{0}\mathcal{B}_{\epsilon}}\,\phi_{\epsilon}\,[\,\nabla p_{\epsilon}\cdot\bm{n}_{\mathcal{B}_{\epsilon}}\,]\,\sigma(d\mu_{\epsilon})\;\;\;\text{and}\\
K_{\pm} & \,=\,-\theta_{\epsilon}\,\epsilon\,\int_{\partial_{\pm}\mathcal{B}_{\epsilon}}\,\phi_{\epsilon}\,[\,\nabla p_{\epsilon}\cdot\bm{n}_{\mathcal{B}_{\epsilon}}\,]\,\sigma(d\mu_{\epsilon})\;.
\end{align*}
First, we show that $K_{0}=o_{\epsilon}(1)$. For $\bm{x}\in\partial_{0}\mathcal{B}_{\epsilon}$,
by \eqref{e_pesB} and \eqref{e_Ces},
\[
\begin{aligned}|\,\nabla p_{\epsilon}(\bm{x})\cdot\bm{n}_{\mathcal{B}_{\epsilon}}(\bm{x})\,| & \,=\,\Big|\,\frac{1}{C_{\epsilon}^{\boldsymbol{\sigma}}}\,e^{-\frac{\mu}{2\epsilon}(\bm{x}\cdot\bm{v})^{2}}\,\bm{v}\cdot\bm{n}_{\mathcal{B}_{\epsilon}}(\bm{x})\,\Big|\,\le\,\frac{C}{\epsilon^{1/2}}\;.\end{aligned}
\]
Therefore, by Proposition \ref{p_bddsol} and \eqref{e_bdE1} along
with the fact that $\partial_{0}\mathcal{B}_{\epsilon}^{\bm{\sigma}}\subset\partial_{0}\mathcal{E}_{\epsilon}^{\bm{\sigma}}$,
we have
\begin{equation}
|\,K_{0}\,|\,\le\,C\,\theta_{\epsilon}\,\epsilon^{1/2}\,Z_{\epsilon}^{-1}\,e^{-(H+J^{2}\delta^{2})/\epsilon}\,\sigma(\partial_{0}\mathcal{B}_{\epsilon})\,\le\,C\,\epsilon^{J^{2}-(d+1)/2}\,\delta^{d-1}\,=\,o_{\epsilon}(1)\label{e_K0}
\end{equation}
for sufficiently large $J$, where we used $\sigma(\partial_{0}\mathcal{B}_{\epsilon})=O(\delta^{d-1})$
in the second inequality. Next, we estimate $K_{+}$ and $K_{-}$.
\begin{lem}
\label{lem_Kpm}We have
\begin{align*}
K_{+} & \,=\,-\frac{\lambda_{1}}{2\pi\nu_{\star}\sqrt{-\det\mathbb{H}}}\,\mathbf{m}_{\epsilon}(i)+o_{\epsilon}(1)\;\;\text{and\;\;}K_{-}\,=\,\frac{\lambda_{1}}{2\pi\nu_{\star}\sqrt{-\det\mathbb{H}}}\,\mathbf{m}_{\epsilon}(j)+o_{\epsilon}(1)\;.
\end{align*}
\end{lem}

\begin{proof}
We only prove the estimate for $K_{+}$ since the proof for $K_{-}$
is identical. Since $\bm{n}_{\mathcal{B}_{\epsilon}}=\bm{e}_{1}$
on $\partial_{+}\mathcal{B}_{\epsilon}$, by the Taylor expansion
of $U$ around $\boldsymbol{\sigma}$, explicit formula \eqref{e_pesB}
for $p_{\epsilon}$, and \eqref{e_Zeps}, we can write
\[
K_{+}\,=\,-[\,1+o_{\epsilon}(1)\,]\,\frac{\epsilon\,\mu^{1/2}\,(\bm{v}\cdot\bm{e}_{1})}{(2\pi\epsilon)^{(d+1)/2}\,\nu_{\star}}\,\int_{\partial_{+}\mathcal{B}_{\epsilon}}\,e^{-\frac{1}{2\epsilon}\bm{x}\cdot\,(\mathbb{H}+\mu\,\bm{v}\otimes\bm{v})\,\bm{x}}\,\phi_{\epsilon}(\bm{x})\,\sigma(d\bm{x})\;.
\]
With the notations introduced in Section \ref{sec72}, we perform
the change of variable $\boldsymbol{y}=\Pi_{\epsilon}(\boldsymbol{x})$
in the previous integral to deduce that
\begin{equation}
K_{+}\,=\,-[\,1+o_{\epsilon}(1)\,]\,\frac{\epsilon\,\mu^{1/2}\,(\bm{v}\cdot\bm{e}_{1})}{(2\pi\epsilon)^{(d+1)/2}\,\nu_{\star}}\,\int_{\Pi_{\epsilon}(\partial_{+}\mathcal{B}_{\epsilon})}\,e^{-\frac{1}{2\epsilon}\bm{y}\cdot\,(\widetilde{\mathbb{H}}+\mu\,\widetilde{\boldsymbol{v}}\otimes\widetilde{\boldsymbol{v}})\,\bm{y}}\,\phi_{\epsilon}(\Pi_{\epsilon}^{-1}(\bm{y}))\,d\bm{y}\;,\label{e_K+}
\end{equation}
where we applied Lemma \ref{lem73} to the exponential term. Let $r>0$
be the constant appearing in Lemma \ref{lem74}. By Proposition \ref{p55}
and Lemma \ref{lem74}, we have $\phi_{\epsilon}(\Pi_{\epsilon}^{-1}(\bm{y}))=\mathbf{m}_{\epsilon}(i)+o_{\epsilon}(1)$
for $\boldsymbol{y}\in\mathcal{D}_{r\delta}^{(d-1)}$. Thus, we have
\[
\int_{\mathcal{D}_{r\delta}^{(d-1)}(\boldsymbol{0})}\,e^{-\frac{1}{2\epsilon}\bm{y}\cdot\,(\widetilde{\mathbb{H}}+\mu\,\widetilde{\boldsymbol{v}}\otimes\widetilde{\boldsymbol{v}})\,\bm{y}}\,\phi_{\epsilon}(\Pi_{\epsilon}^{-1}(\bm{y}))\,d\bm{y}\,=\,\frac{(2\pi\epsilon)^{(d-1)/2}}{\sqrt{\det\,(\widetilde{\mathbb{H}}+\mu\widetilde{\boldsymbol{v}}\otimes\widetilde{\boldsymbol{v}})}}\,[\,\mathbf{m}_{\epsilon}(i)+o_{\epsilon}(1)\,]\;.
\]
Since the integral on $\Pi_{\epsilon}(\partial_{+}\mathcal{B}_{\epsilon})\setminus\mathcal{D}_{r\delta}^{(d-1)}(\boldsymbol{0})\subset\mathbb{R}^{d-1}\setminus\mathcal{D}_{r\delta}^{(d-1)}(\boldsymbol{0})$
is $o_{\epsilon}(1)$ by Proposition \ref{p_bddsol}, we can conclude
from \eqref{e_K+} that
\[
K_{+}\,=\,-\,[\,1+o_{\epsilon}(1)\,]\,\frac{\mu^{1/2}\,(\bm{v}\cdot\bm{e}_{1})}{2\pi\nu_{\star}\sqrt{\det\,(\widetilde{\mathbb{H}}+\mu\widetilde{\boldsymbol{v}}\otimes\widetilde{\boldsymbol{v}})}}\,\mathbf{m}_{\epsilon}(i)+o_{\epsilon}(1)\;.
\]
The proof is completed by the second part of Lemma \ref{lem73}.
\end{proof}
Now, \eqref{e_A1} can be obtained by combining \eqref{e_A1dec},
\eqref{e_K0}, and Lemma \ref{lem_Kpm}.

\subsection{\label{sec74}Estimate of $A_{2}^{\pm}(\boldsymbol{\sigma})$}

Now, we estimate $A_{2}^{+}(\boldsymbol{\sigma})$ and $A_{2}^{-}(\boldsymbol{\sigma})$.
Since the proof is identical, it suffices to consider $A_{2}^{+}(\boldsymbol{\sigma})$,
i.e., \eqref{e_A2+}. Write $\boldsymbol{\ell}=(\ell_{1},\,\dots,\,\ell_{d})$
and $\boldsymbol{v}=(v_{1},\,\dots,\,v_{d})$. Then, by \eqref{e617}
and \eqref{e618}, we can write
\begin{equation}
A_{2}^{+}(\boldsymbol{\sigma})\,=\,M_{1}+M_{2}\;,\label{e_decA2}
\end{equation}
where
\begin{align*}
M_{1} & \,=\,\frac{\theta_{\epsilon}}{\eta C_{\epsilon}}\int_{\widehat{\partial}_{+}\mathcal{B}_{\epsilon}}\phi_{\epsilon}(\bm{x})\,\Big[\,\boldsymbol{x}\cdot\boldsymbol{e}_{1}-\frac{J\delta}{\lambda_{1}^{1/2}}-\eta\,\Big]\,e^{-\frac{\mu}{2\epsilon}(\overline{\boldsymbol{x}}\cdot\bm{v})^{2}}\,\sum_{k=2}^{d}v_{k}\,\ell_{k}(\bm{x})\,\mu_{\epsilon}(d\bm{x})\;\;\;\text{and}\\
M_{2} & \,=\,\theta_{\epsilon}\int_{\widehat{\partial}_{+}\mathcal{B}_{\epsilon}}\phi_{\epsilon}(\bm{x})\,\frac{1-p_{\epsilon}(\overline{\boldsymbol{x}})}{\eta}\,\ell_{1}(\bm{x})\,\mu_{\epsilon}(d\bm{x})\;.
\end{align*}
First, we show that $M_{1}$ is negligible.
\begin{lem}
\label{lem0076}We have that $M_{1}=o_{\epsilon}(1)$.
\end{lem}

\begin{proof}
Since $\big|\boldsymbol{x}\cdot\boldsymbol{e}_{1}-\frac{J\delta}{\lambda_{1}^{1/2}}-\eta\,\big|\le\eta$
for $\boldsymbol{x}\in\widehat{\partial}_{+}\mathcal{B}_{\epsilon}$,
by Proposition \ref{p_bddsol}, it suffices to prove that
\begin{equation}
\frac{\theta_{\epsilon}}{C_{\epsilon}}\,\frac{1}{Z_{\epsilon}}\,\int_{\widehat{\partial}_{+}\mathcal{C}_{\epsilon}}\,e^{-\frac{\mu}{2\epsilon}(\overline{\boldsymbol{x}}\cdot\bm{v})^{2}}\,e^{-\frac{U(\boldsymbol{x})}{\epsilon}}\,d\bm{x}\,=\,o_{\epsilon}(1)\;.\label{e717}
\end{equation}
By applying $U(\boldsymbol{x})=U(\overline{\boldsymbol{x}})+O(\epsilon^{2})$
and then applying the Taylor expansion to $U(\overline{\boldsymbol{x}})$
(with respect to $\boldsymbol{\sigma}=\boldsymbol{0}$), the left-hand
side of the above equality can be bounded from above by
\[
\frac{C}{\epsilon^{(d+1)/2}}\,\int_{\widehat{\partial}_{+}\mathcal{C}_{\epsilon}}\,e^{-\frac{1}{2\epsilon}\overline{\boldsymbol{x}}\cdot(\mathbb{H}+\mu\,\bm{v}\otimes\bm{v})\overline{\boldsymbol{x}}}\,d\boldsymbol{x}\,=\,\frac{C\eta}{\epsilon^{(d+1)/2}}\int_{\partial_{+}\mathcal{C}_{\epsilon}}\,e^{-\frac{1}{2\epsilon}\boldsymbol{x}\cdot\,(\mathbb{H}+\mu\,\bm{v}\otimes\bm{v})\,\boldsymbol{x}}\,\sigma(d\boldsymbol{x})\;.
\]
Using the change of variable $\boldsymbol{y}=\Pi_{\epsilon}(\boldsymbol{x})$
and applying Lemma \ref{lem73}, we can check that the last integral
is bounded by $C\epsilon^{(d-1)/2}$. Hence, the left-hand side of
\eqref{e717} is bounded from above by $\frac{C\eta}{\epsilon^{(d+1)/2}}\times C\epsilon^{(d-1)/2}=o_{\epsilon}(1)$.
This proves the lemma.
\end{proof}
Next, we estimate $M_{2}$.
\begin{lem}
\label{lem76}We have
\begin{equation}
M_{2}\,=\,-\frac{1}{2\pi\nu_{\star}}\,\sqrt{\frac{\mu}{\det(\widetilde{\mathbb{H}}+\mu\widetilde{\boldsymbol{v}}\otimes\widetilde{\boldsymbol{v}})}}\,[\,(\mathbb{L}\,\mathbb{H}^{-1}\,\bm{v})\cdot\bm{e}_{1}\,]\,\mathbf{m}_{\epsilon}(i)+o_{\epsilon}(1)\;.\label{e_M2}
\end{equation}
\end{lem}

\begin{proof}
Let $a\in(0,\,a_{0})$, where $a_{0}$ is the constant appearing in
Lemma \ref{lem_decC}. Let us define
\[
\widehat{\partial}_{+}^{\,2,\,a}\mathcal{B}_{\epsilon}^{\boldsymbol{\sigma}}\,:=\,\widehat{\partial}_{+}\mathcal{B}_{\epsilon}^{\boldsymbol{\sigma}}\,\cap\,\widehat{\partial}_{+}^{\,2,\,a}\mathcal{C}_{\epsilon}^{\boldsymbol{\sigma}}
\]
and write
\begin{equation}
M_{2}\,=\,M_{2,\,1}+M_{2,\,2}\;,\label{e_decM2}
\end{equation}
where
\begin{align*}
M_{2,\,1} & \,=\,\theta_{\epsilon}\int_{\widehat{\partial}_{+}^{\,2,\,a}\mathcal{B}_{\epsilon}}\phi_{\epsilon}(\bm{x})\frac{1-p_{\epsilon}(\overline{\boldsymbol{x}})}{\eta}\ell_{1}(\bm{x})\,\mu_{\epsilon}(\bm{x})\,d\bm{x}\;\;\text{and\;,}\\
M_{2,\,2} & \,=\,\theta_{\epsilon}\int_{\widehat{\partial}_{+}\mathcal{B}_{\epsilon}\setminus\widehat{\partial}_{+}^{\,2,\,a}\mathcal{B}_{\epsilon}}\phi_{\epsilon}(\bm{x})\frac{1-p_{\epsilon}(\overline{\boldsymbol{x}})}{\eta}\ell_{1}(\bm{x})\,\mu_{\epsilon}(\bm{x})\,d\bm{x}\;.
\end{align*}
First, we check that $M_{2,\,1}=o_{\epsilon}(1)$. By Proposition
\ref{p_bddsol}, it suffices to show that $\frac{\theta_{\epsilon}}{\eta}\mu_{\epsilon}(\widehat{\partial}_{+}^{\,2,\,a}\mathcal{B}_{\epsilon})=o_{\epsilon}(1)$.
This is a consequence of the bound $U(\bm{x})\ge H+aJ^{2}\delta^{2}$
on $\widehat{\partial}_{+}^{2,\,a}\mathcal{B}_{\epsilon}$, which
holds by definition, the bound $\text{vol}(\widehat{\partial}_{+}^{2,\,a}\mathcal{B}_{\epsilon})\le\text{vol}(\widehat{\partial}_{+}\mathcal{B}_{\epsilon})\le C\eta\delta^{d-1}$,
and \eqref{e_Zeps}.

Now, we turn to $M_{2,\,2}$. By Lemma \ref{lem_decC}, we have $\overline{\bm{x}}\cdot\bm{v}\ge cJ\delta$
for $\boldsymbol{x}\in\widehat{\partial}_{+}\mathcal{B}_{\epsilon}\setminus\widehat{\partial}_{+}^{2,\,a}\mathcal{B}_{\epsilon}$;
hence, we can use the elementary inequality
\[
\frac{b}{b^{2}+1}e^{-b^{2}/2}\,\leq\,\int_{b}^{\infty}e^{-t^{2}/2}dt\,\leq\,\frac{1}{b}e^{-b^{2}/2}\;\;\;\text{for }b>0
\]
to obtain
\[
1-p_{\epsilon}(\overline{\boldsymbol{x}})\,=\,[\,1+o_{\epsilon}(1)\,]\,\frac{\epsilon^{1/2}}{(2\pi\mu)^{1/2}\,(\overline{\boldsymbol{x}}\cdot\bm{v})}\,e^{-\frac{\mu}{2\epsilon}(\overline{\boldsymbol{x}}\cdot\bm{v})^{2}}\;.
\]
Now, we apply this result along with the Taylor expansions of $U$
and $\boldsymbol{\ell}$ around $\boldsymbol{\sigma}=\boldsymbol{0}$
to $M_{2,\,2}$ to get
\[
M_{2,\,2}\,=\,\frac{1+o_{\epsilon}(1)}{(2\pi)^{(d+1)/2}\,\nu_{\star}\,\mu^{1/2}\,\epsilon{}^{(d+3)/2}}\,\int_{\widehat{\partial}_{+}\mathcal{B}_{\epsilon}\setminus\widehat{\partial}_{+}^{\,2,\,a}\mathcal{B}_{\epsilon}}\,\phi_{\epsilon}(\bm{x})\,\frac{\mathbb{L}\overline{\boldsymbol{x}}\cdot\boldsymbol{e}_{1}}{\overline{\boldsymbol{x}}\cdot\bm{v}}\,e^{-\frac{1}{2\epsilon}\overline{\boldsymbol{x}}\cdot\,(\mathbb{H}+\mu\,\bm{v}\otimes\bm{v})\,\overline{\boldsymbol{x}}}\,d\bm{x}\;.
\]
Note here that we have replaced several $\boldsymbol{x}$'s with $\overline{\boldsymbol{x}}$'s
without changing the error term since $|\overline{\boldsymbol{x}}-\boldsymbol{x}|=O(\eta)$.
Let $r>0$ be the constant appearing in Lemma \ref{lem74}. Then,
we claim that, for all sufficiently small $\epsilon>0$,
\[
\Big(\,\frac{J\delta}{\lambda_{1}^{1/2}},\,\frac{J\delta}{\lambda_{1}^{1/2}}+\eta\,\Big]\,\times\Pi_{\epsilon}^{-1}(\mathcal{D}_{r\delta/2}(\boldsymbol{0}))\,\subset\,\widehat{\partial}_{+}\mathcal{B}_{\epsilon}\cap\widehat{\mathcal{W}}_{i}\,\subset\,\widehat{\partial}_{+}\mathcal{B}_{\epsilon}\setminus\widehat{\partial}_{+}^{\,2,\,a}\mathcal{B}_{\epsilon}\;.
\]
The second inclusion is immediate from the definitions of $\widehat{\mathcal{W}}_{i}$
and $\widehat{\partial}_{+}^{\,2,\,a}\mathcal{B}_{\epsilon}$. On
the other hand, the first inclusion is a consequence of Lemma \ref{lem74}
and the fact that $\eta=\epsilon^{2}$. For convenience, we write
\[
\mathcal{A}_{\epsilon}\,=\,\Big(\,\frac{J\delta}{\lambda_{1}^{1/2}},\,\frac{J\delta}{\lambda_{1}^{1/2}}+\eta\,\Big]\,\times\Pi_{\epsilon}^{-1}(\mathcal{D}_{r\delta/2}(\boldsymbol{0}))\;.
\]

We further decompose $M_{2,\,2}\,=\,M_{2,\,2,\,1}+M_{2,\,2,\,2}$
where $M_{2,\,2,\,1}$ and $M_{2,\,2,\,2}$ are obtained from $M_{2,\,2}$
by replacing the integral $\int_{\widehat{\partial}_{+}\mathcal{B}_{\epsilon}\setminus\widehat{\partial}_{+}^{2,\,a}\mathcal{B}_{\epsilon}}$
with $\int_{(\widehat{\partial}_{+}\mathcal{B}_{\epsilon}\setminus\widehat{\partial}_{+}^{2,\,a}\mathcal{B}_{\epsilon})\setminus\mathcal{A}_{\epsilon}}$
and $\int_{\mathcal{A}_{\epsilon}}$, respectively. We argue that
$M_{2,\,2,\,1}\,=\,o_{\epsilon}(1)$. By Proposition \ref{p_bddsol}
and the fact that $\overline{\bm{x}}\cdot\bm{v}\ge aJ\delta$ on $\widehat{\partial}_{+}\mathcal{B}_{\epsilon}\setminus\widehat{\partial}_{+}^{\,2,\,a}\mathcal{B}_{\epsilon}$,
it suffices to show that
\[
\int_{\mathbb{R}^{d-1}\setminus\Pi_{\epsilon}^{-1}(\mathcal{D}_{r\delta/2}(\boldsymbol{0}))}\,e^{-\frac{1}{2\epsilon}\overline{\boldsymbol{x}}\cdot\,(\mathbb{H}+\mu\,\bm{v}\otimes\bm{v})\,\overline{\boldsymbol{x}}}\,d\widetilde{\bm{x}}\,=\,\epsilon{}^{(d-1)/2}\,o_{\epsilon}(1)\;,
\]
where $\widetilde{\boldsymbol{x}}$ is defined in \eqref{e_tildes}.
The previous identity can be directly verified by the change of variable
$\boldsymbol{y}=\Pi_{\epsilon}(\overline{\boldsymbol{x}})$.

Next, we turn to $M_{2,\,2,\,2}$. On $\mathcal{A}_{\epsilon}$, we
have $\phi_{\epsilon}(\bm{x})=\mathbf{m}_{\epsilon}(i)+o_{\epsilon}(1)$
by Proposition \ref{p55}. Hence, we can write
\begin{align}
M_{2,\,2,\,2}\,=\; & \frac{1+o_{\epsilon}(1)}{(2\pi)^{(d+1)/2}\,\nu_{\star}\,\mu^{1/2}\,\epsilon{}^{(d-1)/2}}\,[\,\mathbf{m}_{\epsilon}(i)+o_{\epsilon}(1)\,]\nonumber \\
 & \times\int_{\Pi_{\epsilon}^{-1}(\mathcal{D}_{r\delta/2}(\boldsymbol{0}))}\,\frac{\mathbb{L}\overline{\boldsymbol{x}}\cdot\boldsymbol{e}_{1}}{\overline{\boldsymbol{x}}\cdot\bm{v}}\,e^{-\frac{1}{2\epsilon}\overline{\boldsymbol{x}}\cdot\,(\mathbb{H}+\mu\,\bm{v}\otimes\bm{v})\,\overline{\boldsymbol{x}}}\,d\widetilde{\boldsymbol{x}}\;.\label{e_M222}
\end{align}
We can use \cite[Lemma 8.11]{LeeSeo} to show that the last integral
can be written as
\[
[\,1+o_{\epsilon}(1)\,]\,\frac{(2\pi\epsilon)^{(d-1)/2}\,(-\mu\,\mathbb{L}\mathbb{H}^{-1}\bm{v})\cdot\bm{e}_{1}}{\sqrt{\textrm{det\,}(\widetilde{\mathbb{H}}+\mu\widetilde{\bm{v}}\otimes\widetilde{\bm{v}})}}\;.
\]
Inserting this into \eqref{e_M222} along with the fact that $M_{2,\,2,\,1}=o_{\epsilon}(1)$
proves that

\begin{equation}
M_{2,\,2}\,=\,-\frac{1}{2\pi\nu_{\star}}\,\sqrt{\frac{\mu}{\det(\widetilde{\mathbb{H}}+\mu\widetilde{\boldsymbol{v}}\otimes\widetilde{\boldsymbol{v}})}}\,[\,(\mathbb{L}\mathbb{H}^{-1}\bm{v})\cdot\bm{e}_{1}\,]\,\mathbf{m}_{\epsilon}(i)+o_{\epsilon}(1)\;.\label{e_M22c}
\end{equation}
Combining this estimate with the fact that $M_{2,\,1}=o_{\epsilon}(1)$
completes the proof.
\end{proof}
Now, the proof of \eqref{e_A2+} follows immediately from \eqref{e_decA2}
and Lemmas \ref{lem73}, \ref{lem0076}, \ref{lem76}. This concludes
the proof of Proposition \ref{p_test}.
\begin{acknowledgement*}
IS and JL was supported by the National Research Foundation of Korea
(NRF) grant funded by the Korea government (MSIT) (No. 2016K2A9A2A13003815,
2017R1A5A1015626 and 2018R1C1B6006896) and the Samsung Science and
Technology Foundation (Project Number SSTF-BA1901-03). The datasets
generated during and/or analysed during the current study are available
from the corresponding author on reasonable request.
\end{acknowledgement*}

\end{document}